\documentclass[11pt]{article}
\usepackage{graphicx}
\usepackage{amssymb,amsfonts,amsthm}
\usepackage{enumerate}
\usepackage{qtree}
\usepackage{color}
\usepackage{subcaption,mathtools}

\usepackage{enumitem}
\usepackage{amsmath,amsthm,amssymb}
\usepackage{amscd}
\usepackage{amsfonts}

\usepackage{xcolor}
\newcommand{\pl}{\mathrm{PLF}}
\newcommand{\zz}{\Z[\frac 12]}

\newcommand{\N}{{\mathbb N}}

\newcommand{\la}{\langle}
\newcommand{\ra}{\rangle}

\newtheorem{theorem}{Theorem}[section]
\newtheorem{lemma}[theorem]{Lemma}

\newtheorem{cy}[theorem]{Corollary}

\theoremstyle{definition}
\newtheorem{df}[theorem]{Definition}

\newtheorem{rk}[theorem]{Remark}
\newtheorem{prob}[theorem]{Problem}

\newcommand{\Z}{\mathbb Z}

\newcommand{\topp}{\mathbf{top}}

\newcommand{\bott}{\mathbf{bot}}

\newcommand{\iv}{^{-1}}

\newcommand{\rr}{{\mathcal R} }

\newcommand{\pp}{{\mathcal P} }

 \usepackage{fancyhdr,lipsum}

\begin{document}
\renewcommand{\theequation}{\thesection.\arabic{equation}}
\bigskip

\title{On closed subgroups of the R. Thompson group $F$}
\author{G. Golan-Polak, M. Sapir\thanks{The first author was partially supported by ISF grant 2322/19. The second author was partially supported by the NSF grant DMS-1901976.}}

\date{}
\maketitle

\begin{abstract}
	We prove that Thompson's group $F$ has a subgroup $H$ such that the conjugacy problem in $H$ is undecidable and the membership problem in $H$ is easily decidable. The subgroup $H$ of $F$ is a \textit{closed} subgroup of $F$.
	 That is, every function in $F$ which is a piecewise-$H$ function belongs to $H$. Other interesting examples of closed subgroups of $F$ include Jones' subgroups $\overrightarrow{F}_n$ and Jones' $3$-colorable subgroup $\mathcal F$. By a recent result of the first author, all maximal subgroups of $F$ of infinite index are closed. In this paper we prove that if $K\leq F$ is finitely generated then the closure of $K$, i.e., the smallest closed subgroup of $F$ which contains $K$, is finitely generated. We also prove that all finitely generated closed subgroups of $F$ are undistorted in $F$. In particular, all finitely generated maximal subgroups of $F$ are undistorted in $F$.
\end{abstract}


\section{Introduction}\label{sec:int}

Recall that R. Thompson's group $F$ is the group of all piecewise-linear homeomorphisms of the interval $[0,1]$ with finitely many breakpoints, where all breakpoints are finite dyadic  (i.e., elements of the set $\mathbb Z[\frac{1}{2}]\cap(0,1)$) and all slopes are integer powers of $2$.
The group $F$ has a presentation with two generators and two defining relations \cite{CFP}. 

Algorithmic problems in $F$ have been extensively studied. It is well known that the word problem in $F$ is decidable in linear time \cite{CFP,SU,GuSa} and that the conjugacy problem is decidable in linear time \cite{GuSa,BM,BHMM}. The simultaneous conjugacy problem \cite{KM}  and twisted conjugacy problem \cite{BMV} have also been proven to be decidable. The article \cite{BBH} gives an algorithm for deciding if a finitely generated subgroup $H$ of $F$ is solvable and in \cite{G}, the first author showed that the generation problem in $F$ is decidable (i.e., there is
 an algorithm for deciding if a finite set of elements of $F$ generates $F$). 
 On the other hand, it is proved in \cite{BMV} that there are orbit undecidable subgroups of $\mathrm{Aut}(F)$ and hence, there are extensions of Thompson’s group $F$ by finitely generated free groups with unsolvable conjugacy problem.

In this paper we consider the conjugacy problem in subgroups of $F$. One of the main results in the paper is the following.

\begin{theorem}\label{mainInt}
	Thompson's group $F$ has a subgroup $H$ such that the conjugacy problem in $H$ is undecidable and the membership problem in $H$ is decidable.
\end{theorem}

The subgroup $H$ from Theorem \ref{mainInt} is a closed subgroup of $F$. Closed subgroups of $F$ were introduced in \cite{GS} (see also \cite{G}), in order to solve Savchuk's problem regarding the existence of non-parabolic maximal subgroups of $F$ of infinite index. They were also essential to the solution of the generation problem in Thompson's group $F$ (see \cite{G}).

Closed subgroups of $F$ can be defined in several different ways. Originally, they were defined by means of a directed $2$-complex associated with them (see Section \ref{sec:core}). They can also be defined as
diagram groups over tree rewriting systems (see Section \ref{sec:closed}). But the simplest description of closed subgroups of $F$ is the following.
\begin{df}
	Let $H$ be a subgroup of $F$. The \emph{closure} of $H$, denoted $Cl(H)$, is the subgroup of $F$ of all piecewise-$H$ functions.
\end{df}

Note that the closure of a subgroup $H$ of $F$ consists of all piecewise-linear functions $f$ from $F$ with finitely many pieces, such that each piece has dyadic endpoints and such that on each piece $f$ coincides with a restriction of some function from $H$. 

\begin{df}\label{def:closed}
	 A subgroup $H$ of $F$ is said to be \emph{closed} if $H=Cl(H)$.
\end{df}

There are many examples of interesting closed subgroups of $F$ (see Section \ref{sec:exa}). Recently, the first author showed that all maximal subgroups of $F$ of infinite index are closed \cite{G21}. Jones' subgroups $\overrightarrow{F}_n$ (see \cite{Jones,GS-J}) and Jones' $3$-colorable subgroup  $\mathcal F$ (see \cite{Ren,AN,Jones1}) are also closed subgroups of $F$.

In this paper we  prove the following.
\begin{theorem}\label{undisInt}
	Let $H$ be a finitely generated subgroup of $F$. Then the closure of $H$ is finitely generated and has linear distortion in Thompson's group $F$.
\end{theorem}

Theorem \ref{undisInt} solves Problems 5.5 and 5.7 in \cite{GS}. Since all maximal subgroups of $F$ of infinite index are closed, Theorem \ref{undisInt} implies that all finitely generated maximal subgroups of $F$ (for instance, the stabilizers in $F$ of rational points of $(0,1)$ \cite{GS2}) are undistorted in $F$. Other applications of Theorem \ref{undisInt} are detailed in Section \ref{sec:exa}.

\vskip .2cm
\textbf{Organization:} In Section \ref{sec:pre} we recall the definition of diagram groups and give two equivalent definitions for closed subgroups of $F$. In Section \ref{sec:con} we prove Theorem \ref{mainInt}.
In Section \ref{sec:semi} we define and study semi-complete rewriting systems in preparation for Section \ref{sec:fin}, where Theorem \ref{undisInt} is proved. In Section \ref{sec:exa} we give examples for subgroups of $F$ which are undistorted in $F$ due to Theorem \ref{undisInt}.  In Section \ref{open} we discuss some open problems.

{\bf Acknowledgements.} The authors are greateful to Benjamin Steinberg for pointing their attention to \cite{CMR}.

\section{Preliminaries}\label{sec:pre}

\subsection{Diagram groups}\label{sec:diagrams}

\begin{df}[Diagrams]
Let $\pp=\la\Sigma\mid R\ra$ be a string rewriting system  (i.e., a semigroup presentation, see, for example \cite[Section 7.1]{Sapir}).

Informally, a diagram $\Delta$ over $\pp$ is a plane directed graph with edges labeled by letters from $\Sigma$ which is a tessellation of a disc, has two special vertices $\iota(\Delta)$ and $\tau(\Delta)$ and two special positive paths $\topp(\Delta),\bott(\Delta)$, both connecting $\iota(\Delta)$ with $\tau(\Delta)$ such that $\Delta$ is situated between $\topp(\Delta)$ and $\bott(\Delta)$. Two diagrams are {\em isotopic} if there exists a smooth deformation of the plane taking one of them to the other.

Formally, for every 
$a\in \Sigma$, let $\epsilon(a)$ be an edge $e$ labeled by $a$. It is a diagram $\Delta$  with vertices $\iota(\Delta)=\iota(e), \tau(\Delta)=\tau(e)$ and paths $\topp(\Delta)=\bott(\Delta)=e$.
For every 
rewriting rule $\ell\to r$ in $R$, let $\Delta(\ell\to r)$ be a planar graph consisting of two directed labeled paths,
the top path labeled by the word $\ell$ and the bottom path labeled by the word $r$,
connecting the same points $\iota(\ell\to r)$ and $\tau(\ell\to r).$
There are three operations that can be applied to diagrams in
order to obtain new diagrams.

(1) {\bf Addition.} Given two diagrams $\Delta_1$ and $\Delta_2$,
one can identify $\tau(\Delta_1)$ with $\iota(\Delta_2).$ The
resulting planar graph is again a diagram denoted by
$\Delta_1+\Delta_2$, whose top (bottom) path is the concatenation
of the top (bottom) paths of $\Delta_1$ and $\Delta_2.$ If
$u=x_1x_2\ldots x_n$ is a word in $X$, then we denote
$\varepsilon(x_1)+\varepsilon(x_2)+\cdots + \varepsilon(x_n)$ (i.e., a simple path labeled by $u$) by $\varepsilon(u)$  and call
this diagram 
 \emph{trivial}\index[g]{diagram!trivial}.

(2) {\bf Multiplication.} If the label of the bottom path of
$\Delta_1$ coincides with the label of the top path of $\Delta_2$,
then we can \emph{multiply} $\Delta_1$ and $\Delta_2$, identifying
$\bott(\Delta_1)$ with $\topp(\Delta_2).$ The new diagram is
denoted by $\Delta_1\circ \Delta_2.$ The vertices
$\iota(\Delta_1\circ \Delta_2)$ and $\tau(\Delta_1\circ\Delta_2)$
coincide with the corresponding vertices of $\Delta_1, \Delta_2$,
$\topp(\Delta_1\circ \Delta_2)=\topp(\Delta_1),
\bott(\Delta_1\circ \Delta_2)=\bott(\Delta_2).$

\begin{figure}
\begin{center} 
\unitlength=1mm
\special{em:linewidth 0.4pt}
\linethickness{0.4pt}
\begin{picture}(124.41,55.00)
\put(1.00,30.00){\circle*{2.00}}
\put(46.00,30.00){\circle*{2.00}}
\put(1.00,30.00){\line(1,0){45.00}}
\bezier{320}(1.00,30.00)(24.00,55.00)(46.00,30.00)
\bezier{332}(1.00,30.00)(24.00,5.00)(46.00,30.00)
\put(24.00,35.00){\makebox(0,0)[cc]{$\Delta_1$}}
\put(24.00,25.00){\makebox(0,0)[cc]{$\Delta_2$}}
\put(24.00,10.00){\makebox(0,0)[cc]{$\Delta_1\circ\Delta_2$}}
\put(66.00,30.00){\circle*{2.00}}
\put(94.00,30.00){\circle*{2.00}}
\put(123.00,30.00){\circle*{2.00}}
\bezier{164}(66.00,30.00)(80.00,45.00)(94.00,30.00)
\bezier{152}(66.00,30.00)(81.00,17.00)(94.00,30.00)
\bezier{172}(94.00,30.00)(109.00,46.00)(123.00,30.00)
\bezier{168}(94.00,30.00)(110.00,15.00)(123.00,30.00)
\put(80.00,30.00){\makebox(0,0)[cc]{$\Delta_1$}}
\put(109.00,30.00){\makebox(0,0)[cc]{$\Delta_2$}}
\put(94.00,10.00){\makebox(0,0)[cc]{$\Delta_1+\Delta_2$}}
\end{picture}
\end{center}
\caption{Operations on diagrams}\label{fig:op}
\end{figure}
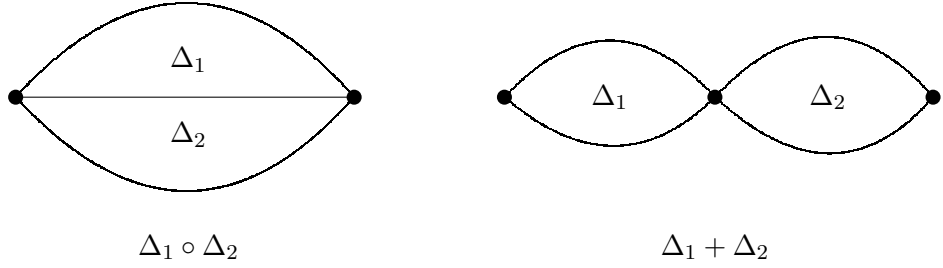

(3) {\bf Inversion.} Given a 
 diagram $\Delta$, we can flip it about
a horizontal line to obtain a new 
 diagram $\Delta\iv$ whose top
(bottom) path coincides with the bottom (top) path of $\Delta.$

A diagram over $\pp$ is any graph obtained from the diagrams  $\epsilon(a)$, $a\in\Sigma$ and $\Delta(\ell\to r)$, $\ell\to r \in R$ by using the  above three operations.


If $u,v$ are words over $\Sigma$
 then a diagram $\Delta$ over $\pp$ is called a $(u,v)$-diagram if the label of $\topp(\Delta)$ is $u$ and the label of $\bott(\Delta)$ is $v$. If $u\equiv v$ (i.e., if $u$ and $v$ are equal letter by letter), the diagram is called {\em spherical}.

A subdiagram of a diagram $\Delta$ which is equal to $\Delta(\ell\to r)^{\pm 1}$, for $\ell\to r\in R$ is called a {\em cell}. Every diagram is a disc (possibly degenerate) tesselated by cells.
\end{df}

Note that if $u$ and $v$ are finite words over $\Sigma$, then there exists a $(u,v)$-diagram over $\mathcal P$ if and only if $u=v$ in the semigroup defined by the presentation $\mathcal P$ \cite{GuSa}.

\begin{df}[Diagram groups] Let $\pp=\la \Sigma\mid R\ra$ be a string rewriting system.

Two cells in a  diagram over $\pp$ form a \emph{dipole}\index[g]{dipole} if the
bottom path of the first  cell coincides with the top path of the
second cell, and the cells are inverses of each other. In this
case, we can obtain a new diagram by removing the two cells and
replacing them by the top path of the first cell. This operation
is called{ \em elimination of dipoles}. The inverse operation is called {\em insertion of dipoles}. We say that two diagrams 
over $\pp$ are \emph{equivalent} 
if one can be obtained from the other by a finite series of elimination and insertion of dipoles.   A
diagram is called \emph{reduced} if it does
not contain dipoles. Every diagram $\Delta$  is equivalent to a 
 unique reduced diagram 
 obtained from $\Delta$ by elimination of dipoles
 \cite{GuSa}.

For every word $w$ the set of all reduced $(w,w)$-diagrams with the product: multiplication followed by reduction is a group denoted by $DG(\pp,w)$ where the identity element is $\epsilon(w)$ and the inverse of $\Delta$ is $\Delta\iv$.
\end{df}
A diagram group $DG(\pp,w)$ is called {\em finitely given} if $\pp$ is finite. As shown in \cite{GuSa} a finitely given diagram group is not necessarily finitely presented or even finitely generated.

\begin{df}[See \cite{GuSa}]\label{Dunce} Let $\rr=\la x\mid x^2\to x\ra$. Then $DG(\rr,x)$ is isomorphic to the R. Thompson group $F$.
\end{df}

\subsection{Squier complexes}

Let $\mathcal P=\la \Sigma\mid R\ra$ be a string rewriting system \cite{Sapir}. With the rewriting  system, one can associate a directed graph $\Gamma$ as follows. The vertex set of $\Gamma$ is the set of all finite words over the alphabet $\Sigma$. Positive edges of $\Gamma$ are tuples of the form $(u,\ell\to r,v)$ where $u,v$ are finite words over $\Sigma$ and $\ell\to r\in R$. The edge $e=(u,\ell\to r,v)$ is an edge  from $u\ell v$ to $urv$. Each positive edge $e=(u,\ell\to r,v)$ has an inverse negative edge $e^{-1}=(u,r\to \ell,v)$ from $urv$ to $u\ell v$.

It is also useful to associate a $2$-complex $K(\mathcal P)$ with the string rewriting system. The 1-skeleton of $K(\mathcal P)$ is the directed graph $\Gamma$. The $2$-cells are $5$ tuples of the form $(u,\ell\to r, q, s\to t, v)$, where $u,\ell,r,q,s,t,v$ are finite words over $\Sigma$  such that $\ell\to r$ and $s\to t$ are rewriting rules from $R$. Such a $2$-cell is glued along the cyclic path $$(u,\ell\to r,qsv)(urq,s\to t,v)(u,r\to\ell,qtv)(u\ell q,t\to s,v).$$
Note that the $2$-cells correspond to independent applications of relations from $R$. Applications of relations $\ell\to r$ and $s\to t$ are called \emph{independent} if the corresponding occurrences of $\ell$ and $s$ do not have common letters (then, the order in which these relations are applied does not affect the result). The $2$-complex $K(\mathcal P)$ is called the \emph{Squier complex} of $\mathcal P$.

For an edge $e$ of $K(\pp)$, we denote by $e_-$ and $e_+$ the initial and terminal vertices of $e$, respectively. A \emph{path} $p$ on $K(\pp)$ is a sequence of edges $p=e_1\cdots e_n$, such that for each $i<n$, we have ${e_i}_+={e_{i+1}}_-$.  The initial vertex of the path is $p_-={e_1}_-$ and the terminal vertex of the path is $p_+={e_n}_+$ (if the path is empty we might refer to any vertex as its initial and terminal vertex). We say that the path $p$ \textit{passes through} all the vertices that edges in the path are incident to. The \emph{length} of the path is the number of edges in the path. We will use the same terminology when considering paths in other directed $2$-complexes (including diagrams over $\mathcal P$).

There is a natural mapping $\delta$ from edges of $K(\mathcal P)$ to diagrams over $\mathcal P$. Let $e=(u,\ell\to r,v)$ be a positive edge of $K(\mathcal P)$. Then $\delta(e)=\epsilon(u)+\Psi+\epsilon(v)$ where $\Psi$ is the elementary $(\ell,r)$-diagram. The inverse edge $(u,r\to\ell,v)$ is mapped to the inverse diagram. The mapping $\delta$ extends to paths on Squier's complex $K(\mathcal P)$. 
Let $p=e_1\cdots e_n$ be a path in $K(\mathcal P)$ then
 $\delta(p)=\delta(e_1)\circ\delta(e_2)\circ \cdots\circ\delta(e_n)$. It is proved in \cite{GuSa} that if $p$ and $p'$ are homotopic paths in $K(\mathcal P)$ then the diagrams $\delta(p)$ and $\delta(p')$ are equivalent. Note also that if $p$ is a path in $K(\mathcal P)$ from a vertex $w$ to itself then $\delta(p)$ is a $(w,w)$-diagram and hence, an element of $DG(\mathcal P,w)$. Hence, for every finite word $w$ over $\Sigma$, $\delta$ induces  a mapping (also denoted by $\delta$) from $\pi_1(K(\mathcal P),w)$ to $DG(\mathcal P,w)$.  By \cite[Theorem 6.1]{GuSa} this mapping is an isomorphism.

\begin{theorem}[\cite{GuSa}]
	Let $\mathcal P=\la \Sigma\mid R\ra$ be a string rewriting system and let $w$ be a finite word over $\Sigma$. Then the diagram group $DG(\mathcal P, w)$ is isomorphic to the fundamental group $\pi_1(K(\mathcal P), w)$.
\end{theorem}

\subsection{Closed subgroups}\label{sec:closed}
\begin{df}
	Let $\mathcal P=\la \Sigma\mid R\ra$ be a string rewriting system. The rewriting system is called a \emph{tree} rewriting system if the following conditions hold
	\begin{enumerate}
		\item[(1)] For any $\ell\to r\in R$, the length $|\ell|=2$ and $|r|=1$.
		\item[(2)] If $\ell_1\rightarrow r_1$ and $\ell_2\rightarrow r_2$ are rewriting rules in $R$, then $\ell_1\equiv \ell_2$ if and only if $r_1\equiv r_2$. In other words, there is at most one rewriting rule with a given left hand-side and at most one rewriting rule with a given right-hand side.
	\end{enumerate}
\end{df}

\begin{lemma}\label{natural}
	Let $\mathcal P=\la \Sigma\mid R\ra$ be a tree  rewriting system. Let $a\in \Sigma$. Then the diagram group $DG(\mathcal P,a)$ naturally embeds into Thompson's group $F$.
\end{lemma}

\begin{proof}
	We define a mapping $\Phi$ from $DG(\mathcal P,a)$ to Thompson's group $F$ as follows. Given a diagram $\Delta$ in $DG(\mathcal P,a)$, we replace the label of every edge in $\Delta$ by $x$. By Definition \ref{Dunce}, the result, $\Phi(\Delta)$, is a diagram in Thompson's group $F$. It suffices to prove that this mapping is injective. To do so, we show that if the diagram $\Delta$ is non-trivial and reduced, then $\Phi(\Delta)$ is also reduced and as such, non-trivial. Assume by contradiction that there is a dipole in $\Phi(\Delta)$. Then there are two cells $\pi_1$ and $\pi_2$ in $\Phi(\Delta)$ such that $\bott(\pi_1)=\topp(\pi_2)$ and such that $\pi_1^{-1}=\pi_2$.
	Consider the cells $\pi_1'$ and $\pi_2'$ in $\Delta$ which correspond to the cells $\pi_1$ and $\pi_2$ in $\Phi(\Delta)$. Since  $\bott(\pi_1')=\topp(\pi_2')$, there is a word $u$ over the alphabet $\Sigma$ such that $\bott(\pi_1')=\topp(\pi_2')$ is labeled by $u$. Let $v_1$ be the label of $\topp(\pi_1')$ and $v_2$ be the label of $\bott(\pi_2')$.
	If $|u|=1$ then $|v_1|=|v_2|=2$ and both $ v_1\to u$ and $v_2\to u$ are rewriting rules in $R$. In that case, since $\mathcal P$ is a tree rewriting system, $v_1\equiv v_2$. Hence, $\pi_1'={\pi_2'}^{-1}$, in contradiction to $\Delta$ being reduced. If $|u|=2$ we get a contradiction in a similar way.
\end{proof}

In view of Lemma \ref{natural} we will often refer to diagram groups $DG(\mathcal P,a)$ where $\mathcal P=\la \Sigma\mid R\ra$ is a tree rewriting system and $a\in \Sigma$ as subgroups of $F$.

\begin{df}\label{closed}
	Let $\mathcal P=\la \Sigma\mid R\ra$ be a tree  rewriting system and let $a\in \Sigma$. Then the diagram group $DG(\mathcal P,a)$, viewed as a subgroup of Thompson's group $F$ (via the natural embedding described in the previous lemma) is called a \emph{closed} subgroup of $F$.
\end{df}



Note that Definition \ref{closed} is equivalent to Definition \ref{def:closed} from the introduction. That is, the following holds.




\begin{lemma}[{\cite[Theorem 1.1]{G}}]\label{dyadic}
	Let $H$ be a subgroup of $F$. Then $H$ is a closed subgroup of $F$ (i.e., $H$  is the natural image of some diagram group $DG(\mathcal P,a)$ where $\mathcal P=\la \Sigma\mid R\ra$ is a tree rewriting system and $a\in\Sigma$) if and only if $H=Cl(H)$\footnote{The definition of $F$ as a group of homeomorphisms of the interval $[0,1]$ and its relation to the definition of $F$ as a diagram group is recalled in Section \ref{sec:exa}.}. 
%
\end{lemma}

We recall some properties of diagrams over tree-rewriting systems.

\begin{df}
	Let $\mathcal P=\la \Sigma\mid R\ra$ be a tree  rewriting system and let $\Delta$ be a diagram over $\mathcal P$. A cell $\pi$ in $\Delta$ is an \emph{expanding cell} if $|\topp(\pi)|=1$ and $|\bott(\pi)|=2$. Otherwise, $\pi$ is called a \emph{reducing cell}.
\end{df}


The following lemma is proved in a more general form in \cite{GuSa}.

\begin{lemma}\label{horizontal}
	Let $\mathcal P=\la \Sigma\mid R\ra$ be a tree  rewriting system. Let $\Delta$ be a reduced spherical diagram over $\mathcal P$. Then there is a unique path $p$ in the diagram $\Delta$ (from $\iota(\Delta)$ to $\tau(\Delta)$) which passes through all the vertices of the diagram. The path $p$ is called the \emph{horizontal path} of $\Delta$. It separates $\Delta$ into two subdiagrams $\Delta^+$ and $\Delta^-$ such that $\Delta=\Delta^+\circ \Delta^-$. The number of cells in $\Delta^+$ is equal to the number of cells in $\Delta^-$ and every cell in $\Delta^+$ is expanding, while every cell in $\Delta^-$ is reducing.
\end{lemma}

It will often be useful to enumerate cells in reduced diagrams $\Delta$ over tree rewriting systems.

\begin{df}[right to left enumeration of cells]\label{rl}
	Let $\mathcal P=\la \Sigma \mid R\ra$ be a tree  rewriting system and let $a\in \Sigma$. Let  $\Delta$ be a reduced $(a,a)$-diagram over $\mathcal P$ and assume that the horizontal path $p$ of $\Delta$ is of length $n$. Then $\Delta=\Delta^+\circ \Delta^-$ and each of the subdiagrams $\Delta^+$ and $\Delta^-$ has $n-1$ cells.
	Let us enumerate the cells of $\Delta^+$ from
	$1$ to $n-1$ by taking every time the ``rightmost'' cell,
	that is, the cell which is to the right of any other cell attached to the
	bottom path of the subdiagram formed by the previous cells.
	The first cell is
	attached to the top path of $\Delta^+$ (which is the top path of $\Delta$).
	For each $i$, let $q_i$ be the bottom path of the subdiagram formed by the first $i-1$ cells (where if $i=1$, we let $q_i=\topp(\Delta^+)$).
	Then the $i^{th}$ cell $\pi_{i}$ in
	this sequence of cells corresponds to an atomic diagram,
	which has the form $\Psi_{i}=(u_{i},r_{i}\to \ell_{i},v_{i})$, such that
	\begin{enumerate}
		\item[(1)] $u_ir_iv_i$ is the label of the path $q_i$.
		\item[(2)] $v_i$
		is the label of longest suffix of $q_i$ which lies on the horizontal path $p$ of $\Delta$.
		\item[(3)] $\pi_i$ is an $(r_i,\ell_i)$-cell. In particular, as $\pi_i$ is an expanding cell,  $\ell_i\to r_i\in R$ and $|\ell_i|=2$.
	\end{enumerate}
	Note that $\Delta^+$ is equal to the composition $\Psi_1\circ \cdots \circ\Psi_{n-1}$.
	In order to enumerate the cells in the diagram $\Delta^{-}$ we consider the inverse diagram $(\Delta^-)\iv$ and enumerate its cells as above. In particular, there are atomic diagrams $\Phi_1,\cdots,\Phi_{n-1}$, such that $(\Delta^-)\iv=\Phi_1\circ\cdots\circ \Phi_{n-1}$ and such that the unique cell in each of the atomic diagrams $\Phi_i$ is an expanding cell.
\end{df}

Since diagram groups $DG(\mathcal P\mid a)$ where $\mathcal P$ is a tree rewriting system and $a\in \Sigma$  naturally embed into Thompson's group $F$, we will often refer to them as subgroups of $F$.
An important property of these subgroups is that there is a simple procedure for determining if a given diagram $\Delta$ in $F$ belongs to such  a given subgroup.

\begin{rk}\label{coloring}
	Let $\mathcal P=\la \Sigma \mid R\ra$ be a tree  rewriting system and let $a\in \Sigma$.
	Let $\Delta$ be a reduced diagram in $F$.
	Then $\Delta$ belongs to the diagram group $G=DG(\mathcal P,a)$ (viewed as a subgroup of $F$) if and only if there is a labeling of the edges of $\Delta$ where each edge is labeled by some letter from $\Sigma$, where $\topp(\Delta)$ and $\bott(\Delta)$ are labeled by $a$ and where every labeled cell is an $(\ell,r)$-cell or an $(r,\ell)$-cell, for some $\ell\to r\in R$.
\end{rk}

Let $\mathcal P=\la \Sigma \mid R\ra$ be a tree  rewriting system and let $a\in \Sigma$.
Let $\Delta$ be a reduced non-trivial diagram in $F$. It follows from Lemma \ref{horizontal} that $\Delta$ is the composition of two subdiagrams $\Delta=\Delta^+\circ \Delta^-$, where every cell in $\Delta^+$ is expanding and every cell in $\Delta^-$ is reducing (indeed, by Lemma \ref{Dunce}, Thompson's group $F$ is a diagram group over the tree rewriting system $\la x\mid x^2\to x\ra$). Let us enumerate the cells in $\Delta^+$ and in $\Delta^-$ from right to left according to the right-to-left order from Definition \ref{rl}. Let $\pi_1,\dots,\pi_n$ and $\pi_1',\dots,\pi_n'$ be the cells of $\Delta^+$ and the cells of $(\Delta^-)\iv$ respectively. To determine if there is a labeling of the edges of $\Delta$ which turns it into an $(a,a)$-diagram over $\mathcal P$, we label the edges of $\Delta^+$ and the edges of $(\Delta^-)\iv$ (viewed as separate diagrams) inductively.
\begin{enumerate}
	\item[(1)] We label $\topp(\pi_1)$ and $\topp(\pi_1')$ by $a$. If there is a relation $bc\to a\in R$ we label the left and right edges of $\bott(\pi_1)$ and $\bott(\pi_1')$ by $b$ and $c$ respectively. If there is no relation in $R$ with $a$ as its right hand side, then $\Delta$ cannot be labeled to become an $(a,a)$-diagram over $\mathcal P$.
	\item[(2)] Let $2\leq i\leq n$ and assume that for all $j<i$ the edges of $\pi_i$ and of $\pi_i'$ were already labeled. By the order the cells were enumerated, $\topp(\pi_i)$ and $\topp(\pi_i')$ were already labeled. If there is a relation in $R$ whose right hand side is the label of $\topp(\pi_i)$ (resp. $\topp(\pi_i')$), we label the bottom edges of $\pi_i$ (resp. $\pi_i'$) with accordance with the left hand side of that relation. If there is no relation whose right hand side is the label of $\topp(\pi_i)$ or no relation whose right hand side is the label of $\topp(\pi_i')$, then $\Delta$ cannot be labeled to become an $(a,a)$-diagram over $\mathcal P$.
\end{enumerate}
If the edges of all cells in $\Delta^+$ and $(\Delta^-)\iv$ are labeled this way and the labels of the edges on the bottom path $\bott(\Delta^+)$ (read from left to right) coincide with the labels of the edges of $\bott((\Delta^-)\iv)$ (read from left to right) then the labeling described turns $\Delta$ into an $(a,a)$-diagram over $\mathcal P$. Otherwise, $\Delta$ is not the image of any diagram from $DG(\mathcal P,a)$ under its natural embedding in Thompson's group $F$. Hence, the described procedure (for a finite tree rewriting system $\mathcal P$) determines if a reduced diagram $\Delta$ belongs to $DG(\mathcal P,a)$. In particular, we have the following.

\begin{theorem}\label{thm:closed}
	Let $\mathcal P=\la \Sigma\mid R\ra$ be a finite tree rewriting system and let $a\in \Sigma$. Then the membership problem in the diagram group $DG(\mathcal P,a)$, viewed as a subgroup of Thompson's group $F$, is decidable.
\end{theorem}

\begin{rk}
	In Theorem \ref{thm:closed} (as well as the rest of the paper), we refer to diagram groups $DG(\mathcal P,a)$ over tree rewriting systems $\mathcal P=\la \Sigma\mid R\ra$, where $a\in\Sigma$. More generally, let $\mathcal P=\la \Sigma\mid R\ra$ be a  tree rewriting system and let $w$ be a finite word over $\Sigma$. It is possible to transform $\mathcal P$ using finitely many Tietze transformations, into a  tree-rewriting system $\mathcal P'=\la \Sigma'\mid R'\ra$, such that $\Sigma\subseteq \Sigma'$, $R\subseteq R'$ and such that there is a letter $a\in \Sigma'$ such that $w=a$ in the semigroup defined by $\mathcal P'$. It follows from \cite[Section 4]{GuSa3} that the diagram group $DG(\mathcal P,w)$ is isomorphic to $DG(\mathcal P',w)$ and by \cite[Theorem 7.1]{GuSa}, the diagram group $DG(\mathcal P',w)$ is isomorphic to $DG(\mathcal P',a)$ (which naturally embeds into Thompson's group $F$). Hence, $DG(\mathcal P,w)$ naturally embeds into Thompson's group $F$ as a closed subgroup. In addition, if $\mathcal P$ (and hence $\mathcal P'$) is finite, the membership problem in this subgroup is decidable.
\end{rk}

\begin{rk} By \cite[Theorem 26]{GubSa1} the derived subgroup $[F,F]$ of $F$ is the diagram group $DG(\mathcal P,a)$ over the following infinite tree rewriting system $$\mathcal P=\la x, a, a_i, b_i, i\ge 0\mid xx=x, a_0b_0=a,a_{i+1}x=a_i, xb_{i+1}=b_i, i=0,1,...\ra.$$
Moreover, using techniques from \cite{G}, one can prove that $[F,F]$ is not the natural embedding into $F$ of any diagram group over a finite tree rewriting system. 
\end{rk}

\subsection{The Stallings $2$-core of subgroups of $F$}\label{sec:core}

Definition \ref{closed}  for closed subgroup of $F$ is different from the original definition given in \cite{GS,G} (but the definitions are equivalent by \cite[Lemma 10.4]{G}). The original definition relied on the construction of the Stallings $2$-core of subgroups of $F$. In this section, we recall the construction and its relation to closed subgroups of $F$. 

Assume that $H$ is a subgroup of $F$. If $H$ is a closed subgroup, we are interested in finding a tree rewriting system $\mathcal P$ such that $H$ is a diagram group over $\mathcal P$. If $H$ is not closed, we are interested in finding the closure of $H$. Note that by Definition \ref{def:closed},
 a subgroup $H$ of $F$ is closed if and only if it contains any function $f\in F$ which is a piecewise-$H$ function.
Hence, the closure of any subgroup $H$ of $F$ is
the smallest closed subgroup of $F$ which contains $H$. 




\begin{df}[The core of a subgroup of $F$]\label{core}
	Let $H$ be a subgroup of $F$ and assume that $\{\Delta_i:i\in \mathcal I\}$ is a set of reduced diagrams generating $H$. We define the \emph{Stallings $2$-core} of $H$ (or the \emph{core} of $H$ for short) to be the directed $2$-complex constructed as follows.
	
	First, we identify all top and bottom edges of the diagrams $\Delta_i$, $i\in\mathcal I$ to a single edge $\rho$. (Here and below, whenever we identify directed edges we also identify their initial vertex to a single vertex and their terminal vertex to a single vertex). The result is a ``bouquet of spheres'' with a directed \textit{distinguished edge} $\rho$. 
	Note that every cell $\pi$ in the bouquet of spheres has two directed boundary components: one boundary component of length $1$ and the other of length $2$.
	We will refer to the boundary component of length $1$ as $\topp(\pi)$ and to the boundary component of length $2$ as $\bott(\pi)$. (Due to the shape of the bouquet of spheres, we cannot use the standard geometric interpretation of the top and bottom paths of a cell, as in Section \ref{sec:diagrams}).
	
	Now, whenever two cells $\pi_1,\pi_2$ in the bouquet of spheres share their top (resp. bottom) boundary component we \emph{fold} the two cells; i.e., we identify the cells  and in particular, we identify each  edge of  $\bott(\pi_1)$ with the respective edge of $\bott(\pi_2)$ (resp. identify $\topp(\pi_1)$ with $\topp(\pi_2)$). 
	
	We apply foldings as long as necessary, until there are no applicable foldings. The directed $2$-complex obtained 
	is the Stallings $2$-core of the subgroup $H$. It is denoted by $\mathcal C(H)$. The initial (resp. terminal) vertex of the distinguished edge $\rho$ is denoted $\iota$ (resp. $\tau$) and is referred to as the \emph{initial vertex} (resp. \emph{terminal vertex}) of the core.
\end{df}

For an example of the construction of the core of a subgroup of $F$, see \cite{GS}.
Note that the Stallings $2$-core of a subgroup $H$ of $F$ does not depend on the chosen generating set, nor on the order of foldings applied \cite{GS}. The following remark follows  from the construction of the core.

\begin{rk}\label{core structure}
	Let $H$ be a subgroup of $F$ and let $\mathcal C(H)$ be the Stallings $2$-core of $H$. Then the initial (resp. terminal) vertex of $\mathcal C(H)$ is the unique vertex of the core which has only outgoing (resp. incoming) edges. The distinguished edge $\rho$ is the unique directed edge from the initial vertex $\iota$ to the terminal vertex $\tau$. In addition, for every vertex $\nu$ in the core there is a directed path in the core from $\iota$ to $\nu$ and a directed path from $\nu$ to $\tau$.
\end{rk}

Now, let $H$ be a subgroup of $F$ and let $\mathcal C(H)$ be the Stallings $2$-core of $H$. The $2$-complex $\mathcal C(H)$ gives rise to a tree  rewriting system (called the \emph{core} rewriting system of the subgroup $H$):
$$\mathcal P=\la E:\bott(\pi)\to\topp(\pi), \pi \mbox{ is a cell in } \mathcal C(H)\ra,$$
where $E$ is the set of edges of $\mathcal C(H)$ (note that the foldings ensure that $\mathcal P$ is indeed a tree rewriting system).
Let $\rho$ be the distinguished edge of $\mathcal C(H)$ (which is also referred to as the \emph{distinguished letter} of the presentation $\pp$) and consider the diagram group $G=DG(\mathcal P,\rho)$ as a subgroup of Thompson's group $F$.
It is easy to see that the subgroup $G$ contains the subgroup $H$. Indeed, if $\mathcal C(H)$ was constructed using the generating set $\{\Delta_i\mid i\in\mathcal I\}$, then it follows from the construction that the edges of each of the diagrams $\Delta_i$ can be labeled to make $\Delta_i$ a $(\rho,\rho)$-diagram over $\mathcal C(H)$. Hence, the generating set of $H$ is contained in $G$. As $G$ is a group, $G$ contains the subgroup $H$.
It was proved in \cite{G} that $G$ is the subgroup of $F$ of all piecewise-$H$ functions. Hence, $G=Cl(H)$ is the smallest closed subgroup of $F$ which contains $H$.

\begin{rk}\label{fin_core}
	Let $H$ be a finitely generated subgroup of $F$. It follows from the construction of the core (starting with a finite generating set of $H$) that the core of $H$ is a finite directed $2$-complex. Note also that for every subgroup $H$ of $F$, the core of $H$ coincides with the core of $Cl(H)$ \cite{GS}.
\end{rk}

\section{The conjugacy problem}\label{sec:con}

In this section we construct a closed subgroup $H$ of $F$ with undecidable conjugacy problem. The closed subgroup $H$ is a diagram group over a finite tree-semigroup presentation. Hence, by Theorem  \ref{thm:closed}, the membership problem in $H$ is decidable. We will need the following lemma.
\begin{lemma} \label{l1} There exists a finite tree rewriting system defining a semigroup which is a group with undecidable word problem.
\end{lemma}

\begin{proof} The first two steps closely follow \cite{CMR}.
	
	{\bf Step 1.} Let $\mathcal G=\{a_1,...,a_m\mid u_1=1,...,u_n=1\}$ be a finite group presentation defining a known example of a group $G$ with undecidable word problem (see, e.g., \cite{Rot}). Then $m>2$ and $n>m$. Consider the balanced presentation
	 $$\mathcal G'=\{a_0,a_1,...,a_m,a_{m+1},\dots,a_n|a_0a_1^2a_2a_3...a_n=1,u_1=1,...,u_n=1\}$$
	Notice that $\mathcal G'$ is obtained from $\mathcal G$ by first adding the generators
	$a_{m+1},\dots,a_n$ to make the presentation balanced; and second, adding the generator
	 $a_0$ along with a relation defining $a_0$ to be the inverse of $a_1^2a_2\dots a_n$ (which does not affect the resulting group).
	 Let $G'$ be the group defined by the presentation $\mathcal G'$. The addition of the generators $a_{m+1},\dots,a_n$  makes $G'$ a free product of the group $G$ and a free group,  so the group $G'$ also has undecidable word problem.
	
	{\bf Step 2.} Let $E=a_0a_1^2a_2a_3...a_n$. Let $E_t=a_1a_2a_3...a_na_0a_1$ and for every $i=0,...,n$ let $E_i^{(l)}$ (resp. $E_i^{(r)}$) be a cyclic shift of $E$ starting (resp. ending) with $a_i$. Note that $E$, $E_t$, $E_i^{(l)}$, $E_i^{(r)}$ are equal to 1 in $G'$.
	Note also that for each $i$, we have $E_i^{(r)}a_i^{-1}=a_i^{-1}$ in the group $G'$.
	For each $j\in\{1,\dots,n\}$, replace every negative letter $a_i\iv$ in the word $u_j$, by the word $E_i^{(r)}a_i^{-1}$.
	 After cancellation, we get a positive word $u_j'$ because for each $i$,
	$E_i^{(r)}$ ends with $a_i$. 
	The presentation
	\begin{equation*}
	\begin{split}
\pp= \{a_0,...,a_n\mid E_ta_0=a_0, E_2^{(l)}u_1'a_1E_2^{(r)}=a_1,
E_3^{(l)}u_2'a_2E_3^{(r)}=a_2,\dots,\\
E_n^{(l)}u_{n-1}'a_{n-1}E_n^{(r)}=a_{n-1}
,E_1^{(l)}u_n'a_nE_1^{(r)}=a_n\}
	\end{split}
	\end{equation*}
	
 is then a semigroup presentation.
	
	By \cite[Proposition 2.1]{CMR} the semigroup defined by this semigroup presentation is  a group which, as is easy to see, is isomorphic to $G'$.
	
	{\bf Step 3.} Let $i=\{0,...,n\}$ and $v\equiv b_1\cdots b_k$, $k>2$ be a semigroup word over the alphabet $\{a_0,\dots,a_n\}$. Consider the following semigroup presentation
	\begin{equation*}
	\begin{split}
\pp(v,a_i)=\{a_0,...,a_n, t_1(v),...,t_{k-2}(v)\mid b_1t_1(v)=a_i, b_2t_2(v)=t_1(v),...,\\ b_{k-2}t_{k-2}(v)=t_{k-3}(v),
b_{k-1}b_k=t_{k-2}(v)\}.
	\end{split}
	\end{equation*}

	Note that $\pp(v,a_i)$ is a tree semigroup presentation. Note also that the semigroup given by $\pp(v,a_i)$  is isomorphic to the semigroup with presentation $\{ a_0,...,a_{n}\mid v=a_i\}$.
	 The isomorphism takes each $a_j$ to itself and each $t_j(v)$ to the element $b_{j+1}...b_k$.

	Now consider the presentation
	\begin{equation*}
	\begin{split}
\pp'=\pp(E_ta_0,a_0)\cup \pp(E_2^{(l)}u_1'a_1E_2^{(r)},a_1)\cup
\pp(E_3^{(l)}u_2'a_2E_3^{(r)},a_2)\cup
...\\	\cup
\pp(E_n^{(l)}u_{n-1}'a_{n-1}E_n^{(r)},a_{n-1})
\cup\pp(E_1^{(l)}u_n'a_nE_1^{(r)},a_n).
	\end{split}
	\end{equation*}  Here for two presentations $\pp_1=\{A_1\mid R_1\}$ and $\pp_2=\{A_2\mid R_2\}$ their union $\pp_1\cup \pp_2$ is defined as $\{A_1\cup A_2\mid B_1\cup B_2\}$.
	Note that in general the union of two tree semigroup presentations is not necessarily a tree semigroup presentation. But here,  $\pp'$ is a tree semigroup presentation, as no two relations in $\pp'$ share their left hand sides and no two relations in $\pp'$ share their right hand side.
	Clearly, the tree semigroup presentation $\pp'$ is a semigroup presentation defining the group $G'$. 
	Hence, the lemma is proved.

\end{proof}

{\em Proof of Theorem \ref{mainInt}}. Consider the tree presentation $\pp'$ from the proof of Lemma \ref{l1}. Let $u,v$ be two words over the generators of $\pp'$. Since the semigroup $G'$ defined by $\pp'$ is a group, there are positive words $p,q$ in the generators of $\pp'$ such that $a_0=pu=qv$ in $G'$. We can assume that $p,q$ end with $a_0$, otherwise we can multiply $p,q$ by $E$ (which is equal to the identity in $G'$) on the right. Now, let $p'$ and $q'$ be the prefixes of $p$ and $q$ respectively such that  $p\equiv p'a_0, q\equiv q'a_0$. 
Since $G'$ is a group, it has an identity element (the word $E$ from the proof of Lemma \ref{l1}). The identity element is an idempotent and it clearly divides $a_0$ (indeed, any two elements in the group divide each other). Hence, by Lemma \cite[Theorem 25]{GubSa1}, the diagram group $DG(\pp',a_0)$ contains a copy of Thompson's group $F$. In particular, the diagram group $DG(\pp',a_0)$ is non-trivial.  Let $\Delta$ be a reduced non-trivial $(a_0,a_0)$-diagram over $\mathcal P'$.
We construct two diagrams $\Delta(u)$ and $\Delta(v)$ as follows. We let
$$\Delta(u)=\Psi\circ (\epsilon(p')+\Delta+\epsilon(u))\circ \Psi^{-1},$$
where $\Psi$ is an $(a_0,p'a_0u)$-diagram (such a diagram exists because $a_0=p'a_0u$ in $G'$). Similarly, we let
$$\Delta(v)=\Phi\circ (\epsilon(q')+\Delta+\epsilon(v))\circ \Phi^{-1}$$
where $\Phi$ is an $(a_0,q'a_0v)$-diagram. 

We will need the following lemma.
\begin{lemma}\label{unde}
	The diagrams $\Delta(u)$ and $\Delta(v)$ are conjugate  in the group $DG(\pp',a_0)$ if and only if $u=v$ in $G'$.
\end{lemma}

\begin{proof}
	The proof is similar to the proof in \cite[Example 15.22]{GuSa}.
	We recall that in \cite{GuSa} two spherical diagrams $\Delta_1$, $\Delta_2$ over the same semigroup presentation $\mathcal P$ were said to be \emph{conjugate diagrams} if there exists a diagram $\eta$ over $\pp$ such that
	$$\Delta_2=\eta^{-1}\circ \Delta_1\circ \eta.$$
	Note that the diagram 
	 $\eta$ in this definition does not have to be a spherical diagram. 
	
	Here, $\Delta(u)$ and $\Delta(v)$ are both $(a_0,a_0)$-diagrams. Thus, if there exists a diagram $\eta$ over $\pp'$ such that $\eta^{-1}\circ \Delta(u)\circ \eta=\Delta(v)$, then $\eta$ must be an $(a_0,a_0)$-diagram. Hence, the diagrams $\Delta(u)$ and $\Delta(v)$ are conjugate as elements of $DG(\pp',a_0)$ if and only if they are conjugate diagrams over $\pp'$. Since the diagram $\Delta(u)$ is a conjugate of the diagram $\epsilon(p')+\Delta+\epsilon(u)$ and the diagram $\Delta(v)$ is a conjugate of the diagram $\epsilon(q')+\Delta+\epsilon(v)$, the diagrams $\Delta(u)$ and $\Delta(v)$ are conjugate if and only if $\epsilon(p')+\Delta+\epsilon(v)$ and $\epsilon(q')+\Delta+\epsilon(v)$ are conjugate as diagrams over $\pp'$.

	Following \cite[Section 15]{GuSa} we say that a  diagram $\eta$ over a semigroup presentation $\mathcal P$ is
	\emph{trivial} if there are no cells in $\eta$ (i.e., if $\eta$ is of the form $\epsilon(w)$ for some word $w$ over $\pp$). We say that $\eta$ is
	\emph{simple} if $\topp(\eta)$ and $\bott(\eta)$ do not have common points, apart from $\iota(\eta)$ and $\tau(\eta)$.
	If $\eta$ is a spherical diagram, we say that it is \emph{absolutely reduced} if for any $n\in\mathbb{N}$, the diagram $\eta^n$, which is the concatenation (without reduction) of $n$ copies of $\eta$ is reduced.

	In \cite{GuSa}, it was noted that if $\eta$ is a spherical diagram, then $\eta$ can be uniquely decomposed into a sum of spherical diagrams $\eta_1+\eta_2+\cdots+\eta_m$ where each summand is either trivial or indecomposable into a sum of spherical diagrams and such that no two consecutive summands are trivials. The summands $\eta_1,\dots,\eta_m$ are called the \emph{components} of $\eta$. 
	By \cite[Lemma 15.14]{GuSa} for every spherical diagram $\eta$, one can effectively find an absolutely reduced 
	spherical diagram $\hat{\eta}$ conjugate to $\eta$.
	
	Now, consider the  $(a_0,a_0)$-diagram $\Delta$. Let $\hat{\Delta}$ be an absolutely reduced 
	spherical diagram conjugate to $\Delta$.
	Clearly, the diagram $\epsilon(p')+\Delta+\epsilon(u)$ is conjugate to the diagram $\epsilon(p')+\hat{\Delta}+\epsilon(u)$. Similarly,
	$\epsilon(q')+\Delta+\epsilon(v)$ is conjugate to the diagram $\epsilon(q')+\hat{\Delta}+\epsilon(v)$. Hence, $\Delta(u)$ and $\Delta(v)$ are conjugate if and only if  $\epsilon(p')+\hat{\Delta}+\epsilon(u)$ and $\epsilon(q')+\hat{\Delta}+\epsilon(v)$ are conjugate.
	
	Let $\hat{\Delta}_1,\dots,\hat{\Delta}_m$, $m\geq 1$ be the components of $\hat{\Delta}$ so that
	$$\hat{\Delta}=\hat{\Delta}_1+\cdots+\hat{\Delta}_m$$
	and note that
	$$(1)\  \epsilon(p')+\hat{\Delta}+\epsilon(u)=\epsilon(p')+\hat{\Delta}_1+\cdots+\hat{\Delta}_m+\epsilon(u),$$
	$$(2)\ \epsilon(q')+\hat{\Delta}+\epsilon(v)=\epsilon(q')+\hat{\Delta}_1+\cdots+\hat{\Delta}_m+\epsilon(v).$$
	
		By \cite[Lemma 15.15]{GuSa}, two absolutely reduced diagrams are conjugate if and only if they have the same number of components and each component of the first diagram (counted from left to right) is conjugate to the corresponding component of the second diagram.
		
		We consider two cases.
		
	Case 1: The diagrams  $\hat{\Delta}_1$ and $\hat{\Delta}_m$ are not trivial.
	
	In that case, the $m+2$ summands on the right hand sides of (1) and (2) are the components of the diagrams $\epsilon(p')+\hat{\Delta}+\epsilon(u)$ and $\epsilon(p')+\hat{\Delta}+\epsilon(u)$, respectively.
	 Hence, the diagrams are conjugate if and only if $\epsilon(p')$ is conjugate to $\epsilon(q')$ and $\epsilon(u)$ is conjugate to $\epsilon(v)$.
	
	Note that $\epsilon(p')$ is conjugate to $\epsilon(q')$ if and only if $p'=q'$ in the group $G'$. Indeed, if $p'=q'$ in $G'$ then there exists a $(p',q')$-diagram $\xi$ over $\pp'$. Then $\xi^{-1}\circ \epsilon(p')\circ\xi=\epsilon(q')$. In the other direction, if $\epsilon(p')$ is conjugate to $\epsilon(q')$ then the conjugating diagram must be a $(p',q')$-diagram over $\pp'$ which implies that $p'=q'$ in $G'$.
	Similarly, $\epsilon(u)$ is conjugate to $\epsilon(v)$ if and only if $u=v$ in $G'$.
	Note also that since $G'$ is a group and $p'a_0u=a_0$ and $q'a_0v=a_0$, the equality $u=v$ holds in $G'$ if and only if $p'=q'$ in $G'$.
	
	Hence, in the case where $\hat{\Delta}_1$ and $\hat{\Delta}_m$ are not trivial, we got that $\Delta(u)$ and $\Delta(v)$ are conjugate if and only if $u=v$ in $G'$.
	
	Case 2: At least  one of the diagrams $\hat{\Delta}_1$ and  $\hat{\Delta}_m$ is trivial.
	
	We assume here that both diagrams  $\hat{\Delta}_1$  and  $\hat{\Delta}_m$ are trivial, the other cases being similar.
	Let $w_1,w_m$ be words such that $\hat{\Delta}_1=\epsilon(w_1)$ and $\hat{\Delta}_m=\epsilon(w_m)$ and note that we must have $m\geq 3$, since otherwise $\hat{\Delta}$ is trivial but conjugate to the reduced non-trivial diagram $\Delta$.
	 Note also that each of the
	 diagrams $\epsilon(p')+\Delta+\epsilon(u)$ and $\epsilon(q')+\Delta+\epsilon(v)$  divides into $m$ components as follows
	   	$$\epsilon(p')+\Delta+\epsilon(u)=\epsilon(p'w_1)+\hat{\Delta}_2\cdots+\hat{\Delta}_{m-1}+\epsilon(w_mu),$$
	   	$$\epsilon(q')+\Delta+\epsilon(v)=\epsilon(q'w_1)+\hat{\Delta}_2\cdots+\hat{\Delta}_{m-1}+\epsilon(w_mv),$$
	   	As above, these diagrams are conjugate if and only if $p'w_1=q'w_1$ in $G'$ and $w_mu=w_mv$ in $G'$. Since $G'$ is a group, these equalities hold in $G'$ if and only if $u=v$ in $G'$.
	   	Hence, in this case we also have that $\Delta(u)$ and $\Delta(v)$ are conjugate if and only if $u=v$ in $G'$.
\end{proof}
Since the word problem in $G'$ is undecidable, it follows from Lemma \ref{unde} that the conjugacy problem in $DG(\pp',a_0)$ is undecidable as well. Since $\pp'$ is a finite tree presentation, the diagram group $DG(\pp',a_0)$ is a closed subgroup of $F$ (recall that $a_0$ is a letter in the alphabet of $\pp'$). Since $\pp'$ is a finite tree-presentation, by Theorem \ref{thm:closed}, the membership problem in this subgroup is decidable.
$\Box$

\section{The Squier complex of semi-complete rewriting systems}\label{sec:semi}

In \cite{GuSa}, Guba and the second author give an algorithm for finding a minimal generating set of the fundamental group of a Squier complex of a complete rewriting system (with respect to any chosen base word).
In this section, we consider more general rewriting systems, which we call ``semi-complete'' rewriting systems (see Definition \ref{semi} below). These rewriting systems are terminating and some of their elements are confluent (i.e., are equivalent to a unique reduced word.) This section follows \cite[Section 9]{GuSa} (with small modifications) to describe a generating set of the fundamental group of  the Squier complex of a semi-complete rewriting system (with respect to a specific base word).




Let $\mathcal P=\la\Sigma\mid R\ra$ be a string rewriting system. Let $\rho$ be a letter in $\Sigma$ and assume that some well order on $\Sigma$ is fixed such that $\rho$ is the least element in $\Sigma$. When we write $u<v$ for two words $u,v$ over the alphabet $\Sigma$, we mean that $u$ is smaller than $v$ in the short-lexicographic order induced on $\Sigma^*$ (i.e., on the set of all finite words over $\Sigma$) by the fixed well order on $\Sigma$. We also make the assumption that for every relation $\ell\to r\in\Sigma$ we have $r<\ell$.

A \emph{derivation} over $\mathcal P$ is a path $p$ 
 on the graph $K(\mathcal P)$. If the path $p$ has initial vertex $u$ and terminal vertex $v$, we say that $p$ is a  derivation  from $u$ to $v$. A \emph{positive derivation} is a derivation which consists entirely of positive edges. That is, a positive derivation from $u$ to $v$ corresponds to an application of a sequence of rewriting rules to the word $u$, such that the word $v$ is obtained at the end of the sequence. 
A word $u$ over $\Sigma$ is said to be \emph{reduced} (with respect to $\mathcal P$) if no rewriting rule of $\mathcal P$ applies to $u$, i.e., if there is no non-trivial positive derivation $p$ starting from the vertex $u$ of $K(\mathcal P)$.
Two words $u,v$ over $\Sigma$ are said to be \emph{equivalent over $\mathcal P$} if there is a derivation from $u$ to $v$ (i.e., if $u=v$ in the semigroup defined by the presentation $\mathcal P$). If $u$ and $v$ are equivalent over $\mathcal P$, we write $u\sim v$. 

We note that for every non-reduced word $u$ there is a positive derivation starting from $u$. Since by assumption, every rewriting rule reduces the ShortLex order of the word it is applied to, every positive derivation over $\mathcal P$ is finite. In other words, $\mathcal P$ is a \textit{terminating} rewriting system.


Often there is more than one positive derivation starting from a word $u$ over $\Sigma$.  Following \cite{GuSa}, we will be interested in two special types of positive derivations.

\begin{df}(Principal left edges)\label{df:principal}
	Let $\mathcal P=\la \Sigma\mid R\ra$ be a rewriting system. Let $e=(u,\ell\to r,v)$ be a positive edge of the Squier complex $K(\mathcal P)$ (that is $\ell\to r\in R$). The edge $e$ is a \emph{principal left edge}\footnote{Note that our definition of  principal left edges differs slightly from the one in \cite{GuSa}. We call a (positive) edge a  principal left edge if and only if its inverse (negative) edge is a principal left  edge in the terminology of \cite{GuSa}.} if the following conditions hold.
	\begin{enumerate}
		\item[(1)] Every proper prefix of $u\ell$ is reduced over $\mathcal P$.
		\item[(2)] $\ell$ is the biggest suffix of $u\ell$ which is the left hand side of some relation in $R$.
		\item[(3)] If $\ell\to r'\in R$, then $r\leq r'$ in the ShortLex order. 
	\end{enumerate}
	A positive derivation $p=e_1\dots e_n$ over $\mathcal P$ is a \emph{left derivation} if for each $i$, the edge $e_i$ is a principal left edge.
	Principal right edges and right derivations are defined in complete analogy.
\end{df}

The following observation regarding  principal left edges is quite useful.

\begin{rk}(The stability property)
	For any $v,v'$ the edge $(u,\ell\to r,v)$ is a  principal left edge if and only if the edge $(u,\ell\to r,v')$ is a  principal left edge
\end{rk}

An analogous stability property holds for  principal right edges.

Note that every non-reduced word $u$ over $\Sigma$, viewed as a vertex of $K(\mathcal P)$, has a unique outgoing principal left (resp. right)  edge. 
Hence, if $u$ is a non-reduced word over $\Sigma$ and $p$ and $q$ are left (resp. right) derivations starting from $u$, then either $p$ is an initial subpath of $q$ or $q$ is an initial subpath of $p$. Let $s$ be a left (resp. right) derivation starting from $u$ and terminating in a reduced word (such a  finite derivation must exist since $\mathcal P$ is terminating). Then $s$ is the longest left (resp. right) derivation starting from $u$ and every other left (resp. right) derivation starting from $u$ must be a prefix of $s$. In particular, there is a unique reduced word over $\mathcal P$ which one can get to via a left (resp. right) derivation starting from $u$. We denote this reduced word by $\bar{u}^\ell$ (resp. by $\bar{u}^r$).
Given two words $u$ and $v$ over $\Sigma$ we write  $v<_\ell u$ (resp. $v<_r u$) if there is a positive left (resp. right) derivation from $u$ to $v$. Note that the relations $<_\ell$ and $<_r$ are transitive. By the stability property, the relation $<_\ell$ is right-invariant (i.e. if $v<_\ell u$ then for every word $w$ over $\Sigma$ we have $vw<_\ell uw$) and the relation $<_r$ is left-invariant. In addition, if $v<_\ell u$ then $\bar{u}^\ell\equiv \bar{v}^\ell$. Similarly, if $v<_r u$ then $\bar{u}^r\equiv \bar{v}^r$.

Let $u,v$ be words in the alphabet $\Sigma$. If $v<_\ell u$ then there is a unique left derivation from $u$ to $v$. We denote this derivation by $d^\ell(u,v)$.  Similarly, if $v<_r u$ we let $d^r(u,v)$ be the unique right derivation from $u$ to $v$. Given an edge $e=(u,\ell\to r,v)$ of $K(\mathcal P)$ and a word $w\in \Sigma^*$ we let $w*e=(wu,\ell\to r,v)$ and $e*w=(u,\ell\to r,vw)$.
If $p=e_1\cdots e_n$ is a path in $K(\mathcal P)$ from $u$ to $v$ we let $w*p=(w*e_1)\cdots(w*e_n)$. Note that $w*p$ is a path in $K(\mathcal P)$ from $wu$ to $wv$. The path $p*w$ is defined in a similar way. Clearly, if $v<_\ell u$ then for any word $w$, the derivation $d^\ell(u,v)*w$ coincides with the left derivation $d^{\ell}(uw,vw)$ from $uw$ to $vw$. A similar observation holds for right derivations.

In this section, we will be interested in rewriting systems where  certain words have a unique equivalent reduced word. Note that if a word $u\in\Sigma^*$ has a unique equivalent reduced word  then $\bar{u}^\ell=\bar{u}^r$. In which case, we will write $\bar{u}$ instead of $\bar{u}^\ell$ and $\bar{u}^r$. Recall that if $u$ and $v$ are words over $\Sigma$, we say that $u$ is a \emph{left divisor} (resp. \emph{right divisor}) of $v$ if there is a word $w\in\Sigma^*$ such that $uw\sim v$ (resp. $wu\sim v$).

\begin{df}\label{semi}
	Let $\mathcal P=\la \Sigma\mid R \ra$ be a rewriting system. Assume that a well order is fixed on $\Sigma$ such that for every rewriting rule in $R$ the right hand-side is smaller than the left hand-side in the ShortLex order. Let $\rho\in\Sigma$ be the least letter in the well order.
	We say that the rewriting system $\mathcal P$ is \emph{semi-complete} (with respect to that fixed well order) if for every 
	 left or right divisor $u$ of $\rho$ there is a unique reduced word over $\mathcal P$ equivalent to $u$.
\end{df}

For the remainder of this section, we assume that the presentation $\mathcal P$ is a semi-complete rewriting system (with respect to the well-order fixed on $\Sigma$ at the beginning of the section).




Let $T$ be the set of all principal left edges in $K(\mathcal P)$. Since every vertex in $K(\mathcal P)$ has at most one outgoing edge from $T$, there are no simple cycles in $K(\mathcal P)$ which consist entirely of edges from $T\cup T^{-1}$. Indeed, assume by contradiction that  $c=e_1^{\epsilon_1}\cdots e_n^{\epsilon^n}$ is a simple cycle in $K(\mathcal P)$, where $e_1,\dots,e_n\in T$ and $\epsilon_1,\dots,\epsilon_n\in\{-1,1\}$. If $\epsilon_1,\dots,\epsilon_n$ are all positive, then $c$ is a ``cyclic'' derivation, which contradicts the fact that $\mathcal P$ is terminating. If $\epsilon_1,\dots,\epsilon_n$ are all negative, we get a contradiction by considering $c^{-1}$. Hence, by applying a cyclic shift to $c$ if necessary, we can assume that $\epsilon_1=-1$ and $\epsilon_{2}=1$. Then the vertex $(e_1^{\epsilon_1})_+=(e_{2}^{\epsilon_{2}})_-$ has two outgoing  principal left edges; a contradiction.

Let $\mathcal C$ be the connected component of $\rho$ in the graph $K(\mathcal P)$. If $u$ is a vertex of $\mathcal C$ then $u\sim \rho$ over $\mathcal P$ and in particular, $u$ is both a left and a right divisor of $\rho$. Hence, $u$ has a unique word equivalent to it over $\mathcal P'$, which must be $\rho$. As such, any positive derivation which starts from $u$ and terminates at a reduced word, will terminate at $\rho$ and in particular, such a positive derivation exists.

Let $\mathcal T_\mathcal{C}$ be the subgraph of $K(\mathcal P)$ whose vertex set is the vertex set of $\mathcal C$ and whose edge set is the set of edges in $T$ whose end vertices belong to $\mathcal C$. We claim that the graph $\mathcal T_{\mathcal C}$ (viewed as an undirected graph) is a spanning tree of the connected component $\mathcal C$. Indeed, since there are no simple cycles which consist of edges from $T\cup T^{-1}$ the graph $\mathcal T_{\mathcal C}$ is cycle-free. Since for every vertex $u$ in $\mathcal C$ there is a left derivation from $u$ to $\rho$, all vertices in $\mathcal C$ are connected to $\rho$ in the subgraph $\mathcal T_{\mathcal C}$. Hence, $\mathcal T_\mathcal{C}$ is a spanning tree of $\mathcal C$.

We are interested in finding a generating set of the fundamental group $\pi_1(K(\mathcal P),\rho)$ (or equivalently, of $\pi_1(\mathcal C,\rho)$). 
 Recall that if $u$ is a vertex of $\mathcal C$ we denote by $d^\ell(u,\rho)$ the left derivation from $u$ to $\rho$ (which is the unique simple path on the tree $\mathcal T_\mathcal C$ from $u$ to $\rho$).


\begin{df}
	Let $e$ be an edge in the connected component $\mathcal C$. Let $u=e_-$ and $v=e_+$. The path
	$(d^\ell(u,\rho))\iv e(d^\ell(v,\rho))$ is a path from $\rho$ to itself, and as such, it is an element of $\pi_1(\mathcal C,\rho)$. We will denote the equivalence class of this path by $[e]$.	If $e_1\cdots e_n$ is a path in the connected component $\mathcal C$ then we let $[e_1\cdots e_n]=[e_1]\cdots[e_n]$. Note that this is the equivalence class of the path $(d^\ell((e_1)_-,\rho))\iv\cdot  e_1\cdots e_n \cdot d^\ell((e_n)_+,\rho)$ in $\pi_1(\mathcal C,\rho)$.
\end{df}

Since $\mathcal T_\mathcal{C}$ is a spanning tree of $\mathcal C$, the standard construction of the fundamental group of a $2$-complex shows that the set of all equivalence classes $[e]$, for positive edges $e$ in $\mathcal C$ which are not principal left edges, is a generating set of $\pi_1(\mathcal C,\rho)$. Moreover, for every principal left edge $e$ in $\mathcal C$, the element $[e]$ is the trivial element of $\pi_1(\mathcal C,\rho)$. We follow Guba and  the second author's method from \cite{GuSa} for constructing a minimal generating set for $\pi_1(\mathcal C,\rho)$. In \cite{GuSa} the method is presented for complete rewriting systems, but it works for semi-complete rewriting systems as well.

Note that if $e=(u,\ell\to r,v)$ is an edge in the connected component $\mathcal C$ then $u\ell v\sim \rho$. Hence, $u$ and $v$ are left and right divisors of $\rho$, respectively. In that case, there are unique reduced words $\bar{u}$ and $\bar{v}$ equivalent to $u$ and $v$ over $\mathcal P$, respectively. Given a non-reduced word $u$, we denote by $e^\ell_u$ (resp. $e^r_u$) the unique outgoing principal left (resp. right) edge of $u$ (note that this edge is the first edge in any left (resp. right) derivation starting from $u$). 
The following lemma (for the case where $\mathcal P$ is complete) follows from the proof of \cite[Theorem 9.5]{GuSa}. We give a proof for completeness.

\begin{lemma}\label{red_edges}
	Let $e=(u,\ell\to r,v)$ be a positive edge in the connected component $\mathcal C$. Then the following assertions hold.
	\begin{enumerate}
		\item[(1)]  $[(u,\ell\to r,v)]=[(\bar{u},\ell\to r,v)]$
		\item[(2)] If $v$ is not reduced, let $e^r_v$ be the unique outgoing principal right edge of $v$.  Then
		$$[(u,\ell\to r,v)]=[(u,\ell\to r,(e^r_v)_+)]^{[((u\ell)*e^r_{v})\iv]}$$
	\end{enumerate}
\end{lemma}

\begin{proof}
	
	(1) If $u$ is reduced, we are done. Hence, assume that $u$ is not-reduced and let $e^\ell_u$ be the unique outgoing principal left edge of $u$. 
	 Let $w$ be the end vertex of $e^\ell_u$ (so that $w<_\ell u$).
	 It suffices to prove that
	 $[(u,\ell\to r,v)]=[(w,\ell\to r,v)]$.
	Let $e^\ell_u=(p,\ell'\to r',q)$ be the  principal left edge from $u$ to $w$, so that $u\equiv p\ell' q$ and $w\equiv pr'q$.
	Note that by the definition of $2$-cells in $K(\mathcal P)$, the following relation holds in $\pi_1(K(\mathcal P),\rho)$.
	$$[(p\ell'q,\ell\to r,v)][(p,\ell'\to r', qr v)]=[(p,\ell'\to r',q\ell v)][(pr 'q,\ell\to r, v)].$$
	Since, $(p,\ell'\to r', q\ell v)$ and $(p,\ell'\to r',qr v)$ are  principal left edges (by the stability property), the corresponding equivalence classes are trivial. Hence
	$$[(p\ell'q,\ell\to r,v)]=[(pr 'q,\ell\to r, v)].$$
	or in other words
	$$[(u,\ell\to r,v)]=[(w,\ell\to r,v)]$$  as required.

	(2) Let $e^r_v=(p,\ell'\to,r',q)$ such that $\ell'\to r'\in R$, and $p\ell'q\equiv v$. By the construction of $2$-cells in $K(\mathcal P)$ we have
		$$[(u,\ell\to r,p\ell'q)][(ur p,\ell'\to r', q)]=[(u\ell p,\ell'\to r',q)][(u,\ell\to r, pr'q)].$$
	By part (1), we have $[(ur p,\ell'\to r', q)]=[(u\ell p,\ell'\to r',q)]$ since $\overline{ur p}=\overline{u\ell p}$. Note also that $(u\ell p,\ell'\to r', q)=(u\ell)*e_r^v$. In addition, $pr'q\equiv (e^r_v)_+$ and $p\ell'q\equiv v$. Hence,
	$$[(u,\ell\to r,v)][(u\ell)*e_r^v]=[(u\ell)*e_r^v][(u,\ell\to r, (e^r_v)_+)].$$
	Clearly, the result follows.  
\end{proof}

As a corollary of Lemma \ref{red_edges} we have the following,

\begin{cy}\label{red}
	Let $e=(u,\ell\to r,v)$ be a positive edge in the connected component $\mathcal C$. Let $d^r(v,\bar{v})$ be the right derivation from $v$ to $\bar{v}$. Then
	$$[(u,\ell\to r,v)]=[(\bar{u},\ell\to r,\bar{v})]^{[((u\ell)*d^r(v,\bar{v}))\iv]}$$
\end{cy}

\begin{proof}
	Follows from a series of applications of Lemma \ref{red_edges}(2) and a single application of Lemma \ref{red_edges}(1).
\end{proof}

The following (for the case where $\mathcal P$ is complete) follows from \cite[Theorem 9.8]{GuSa}. 


%
%

\begin{theorem}\label{X}
	Let $\mathcal P=\la \Sigma\mid R\ra$ be a semi-complete rewriting system (with respect to some well-order on $\Sigma$). Let $\rho\in\Sigma$ be the smallest letter in $\Sigma$. Let $\mathcal C$ be the connected component of $\rho$ in the Squier complex $K(\mathcal P)$. Let $B$ be the set of all positive edges $e=(u,\ell\to r,v)$ in $K(\mathcal P)$ which satisfy the following conditions
	\begin{enumerate}
		\item[(1)] $u\ell v\sim \rho$ over $\mathcal P$ (i.e., the edge $e$ is in the connected component $\mathcal C$).
		\item[(2)] $u$ and $v$ are reduced over $\mathcal P$.
		\item[(3)] $e$ is not a principal left edge.
	\end{enumerate}
	Let $X$ be the set of all equivalence classes $[e]$ for $e\in B$. Then $X$ is a minimal
	generating set of $\pi_1(\mathcal C,\rho)$ of smallest possible cardinality.
\end{theorem}

\begin{proof} 
	The proof requires slight modifications of the arguments from \cite[Section 9]{GuSa}. We only explain here why $X$ is a generating set of $\pi_1(K(\mathcal P),\rho)$ (as this is the only part of the theorem we  use  below).
	Let $e=(u,\ell\to r, v)$ be a positive edge in the connected component $\mathcal C$. It suffices to prove that $[e]$ is generated by $X$.
	Let $d^r(v,\bar{v})$ be the right derivation from $v$ to $\bar{v}$. By Corollary \ref{red},
	$$(*)\ \ [(u,\ell\to r,v)]=[(\bar{u},\ell\to r,\bar{v})]^{[((u\ell)*d^r(v,\bar{v}))\iv]}.$$
	Note that since $d^r(v,\bar{v})$, is a right derivation, for every edge of $(d^r(v,\bar{v}))\iv$ the third coordinate is reduced. Clearly, this remains true for every edge of the path $((u\ell)*d^r(v,\bar{v}))\iv$
	Hence, all edges on the right hand side of $(*)$ have reduced third coordinates. By Lemma \ref{red_edges}(1) we can assume that all edges on the right hand side of $(*)$ have reduced first coordinates as well. Then for each edge $e'$ on the right hand side of $(*)$, $e'$  is either a principal left edge or the inverse of a  principal left edge (in which case, $[e']=1$) or it belongs to $B\cup B^{-1}$ (in which case $[e']\in X\cup X^{-1}$). Hence, $[e]$ is generated by elements of $X$, as required.
\end{proof}


The next observation is clear from the proof of Theorem \ref{X}.

\begin{lemma}\label{length}
	Let $\mathcal P=\la \Sigma\mid R\ra$ be a semi-complete rewriting system (with respect to some well-order on $\Sigma$). Let $\rho\in\Sigma$ be the smallest letter in $\Sigma$ and let $X$ be the generating set of $\pi_1(K(\mathcal P),\rho)$ from Theorem \ref{X}. Let $v$ be a right divisor of $\rho$. Then the word length of $[d^r(v,\bar{v})]$ with respect to $X$ is bounded from above by the length (i.e., the number of edges) of the derivation $d^r(v,\bar{v})$ viewed as a directed path in $K(\mathcal P)$.
\end{lemma}



\section{Closed subgroups of $F$ with finite core are   finitely generated and have linear distortion in $F$}\label{sec:fin}

In this section we prove that if $H$ is a closed subgroup of $F$ whose core is a finite directed $2$-complex, then $H$ is finitely generated and undistorted in $F$. Note that by  \cite[Corollary 10.8]{G}, if  $H$ is closed and its core is a finite directed $2$-complex then there is a finitely generated subgroup $K$ of $F$ such that $H=Cl(K)$ (i.e., such that the core of $K$ coincides with the core of $H$). Here we prove that if $H$ is closed and has a finite core, then $H$ itself must be finitely generated. 

Let $H$ be a subgroup of $F$ with core $\mathcal C(H)$.
 Let $\rho$ be the distinguished edge of the core and let $\iota$ and $\tau$ be the initial and terminal vertices of the core, respectively.
Let $E$ be the set of directed edges of the core and let $R$ be the set of all relations of the form $\bott(\pi)\to\topp(\pi)$ for all cells $\pi$ in the core $\mathcal C(H)$. Let $\mathcal P=\la E\mid R\ra$ be the core rewriting system of $H$. 
Note that every word $w=e_1\cdots e_n$ over $E$ can be viewed as a word over the alphabet of $\mathcal P$ and 
as a sequence of directed edges in the core $\mathcal C(H)$. We will say that $w$ is a \emph{directed path} in the core if for every $i<n$ the terminal vertex of the edge ${e_i}$ coincides with the initial vertex of ${e_{i+1}}$.

We make use of the following theorem. 

\begin{theorem}[{\cite[Propositions 10.1 and 10.2]{G}}]\label{paths}
	Let $H$ be a closed subgroup of $F$ and let $\mathcal C(H)$ be the core of $H$.	Let $\mathcal P=\la E\mid R\ra$ be the core rewriting system and let $\rho\in E$ be the distinguished edge of the core. Let $w,w_1,w_2$ be words over the alphabet $E$. Then the following assertions hold.
	
	\begin{enumerate}
				
		\item[(1)] The element $w$ is a left divisor of   $\rho$ if and only if $w$ is a directed path in the core $\mathcal C(H)$ with initial vertex $\iota$.
		
		\item[(2)] The element $w$ is a right divisor of   $\rho$ if and only if $w$ is a directed path in the core $\mathcal C(H)$ with terminal vertex $\tau$.
		
		\item[(3)] Assume that $w_1$ and $w_2$ are left  divisors of $\rho$, (so that they are both directed paths in the core $\mathcal C(H)$ with initial vertex $\iota$). Then $w_1\sim w_2$ if and only if $(w_1)_+=(w_2)_+$ (i.e., if and only if the directed paths $w_1$ and $w_2$ in the core $\mathcal C(H)$ have the same terminal vertex).
		
		\item[(4)]  Assume that $w_1$ and $w_2$ are right  divisors of $\rho$, so that they are both directed paths in the core $\mathcal C(H)$ with terminal vertex $\tau$. Then $w_1\sim w_2$ if and only if $(w_1)_-=(w_2)_-$ (i.e., if and only if the directed paths $w_1$ and $w_2$ in the core $\mathcal C(H)$ have the same initial vertex).
	\end{enumerate}
\end{theorem}

Note that by Theorem \ref{paths} a word $w$ over $E$ satisfies $w\sim\rho$ over $\mathcal P$ if and only if $w$ is a directed  path in the core $\mathcal C(H)$ with initial vertex $\iota$ and terminal vertex $\tau$.

Now, assume that some well order is fixed on the alphabet $E$, so that $\rho$ is the least element of $E$ in that order. Let $\nu$ be a vertex of the core. Let $P_{\iota,\nu}$ be the set of all directed paths from $\iota$ to $\nu$ in the core  $\mathcal C(H)$ and note that by Remark \ref{core structure}, this set is necessarily non-empty. (For $\nu=\iota$, this set only contains the trivial (empty) path from $\iota$ to itself). We denote by $\iota_{\nu}$ the smallest word in $P_{\iota,\nu}$ with respect to the ShortLex order on $E^*$ induced by the well order on $E$. Similarly, for every $\nu$ we consider the set $P_{\nu,\tau}$ of all directed paths in the core from $\nu$ to $\tau$. We denote by $\tau_\nu$ the smallest word in $P_{\nu,\tau}$ in the short-lexicographic order. Note also that $\iota_\tau=\tau_\iota=\rho$ and that $\iota_\iota$ and $\tau_\tau$ are empty paths.

We define two sets of rewriting rules $u\to v$ for which $u\sim v$ over $\mathcal P$.
The set $R_{\iota}$ of \emph{initial rewriting rules} (i.e., rewriting rules associated with paths on the core starting from the initial vertex $\iota$) 
 is defined as follows.
 \begin{equation*}
 \begin{split}
 R_\iota=\{ \iota_{e_-} e\to \iota_{e_+}\mid \mbox{ $e$ is an edge of $\mathcal C(H)$ such that $\iota_{e_-} e\not\equiv \iota_{e_+}$}
 \}
 \end{split}
 \end{equation*}
(Note that for every rewriting rule in $R_\iota$ the left and right hand sides are indeed equivalent over $\mathcal P$ since they are both paths with initial vertex $\iota$ and the same terminal vertex).
Similarly, the set $R_\tau$ of \emph{terminal rewriting rules} is defined as follows.
 \begin{equation*}
\begin{split}
R_\tau=\{ e\tau_{e_+} \to \tau_{e_-}\mid \mbox{ $e$ is an edge of $\mathcal C(H)$ such that $e\tau_{e_+} \not\equiv \tau_{e_-}$}
\}
\end{split}
\end{equation*}


Now, let $R'=R\cup R_\iota\cup R_\tau$ and consider the rewriting system $\mathcal P'=\la E\mid R'\ra$.
Note that for any two words $u,v\in E^*$, the words $u$ and $v$ are equivalent over $\mathcal P$ if and only if they are equivalent over $\mathcal P'$.
The next three lemmas show that $\mathcal P'$ is a semi-complete rewriting system. 

\begin{lemma}\label{term}
	The following assertions hold for the rewriting system $\mathcal P'=\la E\mid R'\ra$.
	\begin{enumerate}
		\item[(1)] 	Let $u\to v$ be a rewriting rule in $R'$. Then $v$ is smaller than $u$ in the ShortLex order on $E^*$.
		\item[(2)] For every vertex $\nu$ of the core $\mathcal C(H)$, the words $\iota_\nu$ and $\tau_\nu$ are reduced over $\mathcal P'$.
	\end{enumerate}
\end{lemma}

\begin{proof}
	(1) First note that every rewriting rule $\ell\to r$ in $R$ is length decreasing. Indeed, if $\ell\to r\in R$, then $|\ell|=2$ and $|r|=1$.
	It remains to consider rewriting rules in $R_\iota$ and $R_\tau$. Let $e$ be an edge such that $\iota_{e_-}e\to \iota_{e_+}$ is a rewriting rule in $R_\iota$. Then $\iota_{e_-}e\not\equiv \iota_{e_+}$ and both $\iota_{e_-}e$ and $\iota_{e_+}$ are paths in $P_{\iota,e_+}$. By the choice of $\iota_{e_+}$, its ShortLex order is smaller than that of $\iota_{e_-}e$. The argument for rewriting rules in $R_{\tau}$ is similar.
	
	(2) By part (1), every rewriting rule decreases the ShortLex order of any word it is applied to. Assume by contradiction that there is a vertex $\nu$ for which $\iota_\nu$ is not reduced over $\mathcal P'$. Then $\iota_\nu$ is equivalent over $\mathcal P'$ (equivalently, equivalent over $\mathcal P$) to a word with a smaller ShortLex order, in contradiction to the choice of $\iota_\nu$. (Note that the set $P_{\iota,\nu}$ defined above is the set of all words $w$ equivalent to $\iota_\nu$ over $\mathcal P$.). Hence, $\iota_\nu$ must be reduced. The argument for $\tau_\nu$ is similar.
\end{proof}





\begin{lemma}\label{left divisor}
	Let $w$ be a left divisor of $\rho$ over $\mathcal P'$. Then $w\sim \iota_{w_+}$ and if $w$ is reduced over $\mathcal P'$ then $w\equiv \iota_{w_+}$. In other words, the unique reduced word over $E$, equivalent to $w$ over $\mathcal P'$, is $\iota_{w_+}$.
\end{lemma}

\begin{proof}
	Since $w$ and $\iota_{w_+}$ are both paths on the core $\mathcal C(H)$ with initial vertex $\iota$ and terminal vertex $w_+$, by Theorem \ref{paths}, we have $w\sim\iota_{w_+}$ over $\mathcal P'$. By Lemma \ref{term}, $\iota_{w_+}$ is reduced over $\mathcal P'$. Hence, it suffices to prove that if $w$ is reduced then $w\equiv \iota_{w_+}$. If $w$ is the empty word, $w\equiv\iota_\iota$ and we are done. Hence, let $e$ be the last edge in the directed path $w$ and let $u$ be such that $w\equiv ue$. Since $w$ is reduced over $\mathcal P'$, its prefix $u$ is also reduced over $\mathcal P'$. Hence, by induction, $u\equiv \iota_{u_+}$. Since $ue\equiv \iota_{u_+}e\equiv \iota_{e_-}e$, the relation $w\equiv ue\equiv \iota_{e_-}e\to \iota_{e_+}\equiv \iota_{w_+}$ belongs to $R_\iota$, unless its right-hand side coincides (letter by letter) with its left hand-side. Since $w$ is reduced by assumption (and cannot be the left hand side of any relation in $R_\iota$) we must have $w\equiv \iota_{w_+}$.
\end{proof}

The following lemma is a right-left analog of Lemma \ref{left divisor}.

\begin{lemma}\label{right divisor}
	Let $w$ be a right divisor of $\rho$ over $\mathcal P$. Then $w\sim \tau_{w_-}$ and if $w$ is reduced over $\mathcal P'$ then $w\equiv \tau_{w_-}$. In other words, the unique reduced word over $E$, equivalent to $w$ over $\mathcal P'$, is $\tau_{w_-}$.\qed
\end{lemma}

Lemmas \ref{term},\ref{left divisor},\ref{right divisor} imply that $\mathcal P'=\la E\mid R'\ra$ is a semi-complete rewriting system (with respect to the fixed well order on $E$).

\begin{df}
	Let $P=\la E\mid R\ra$ be a core rewriting system and let $\rho$ be the distinguished edge of the core. If one fixes a well-order on $E$ (such that $\rho$ is the least element in that order), one can follow the above construction to get the sets of rewriting rule $R_\iota$ and $R_\tau$ as above.
	We will refer to the rewriting system $\mathcal P'=\la E\mid R'=R\cup R_\iota\cup R_\tau\ra$ as the \emph{semi-completion} of the core presentation $\mathcal P$ (with respect to the fixed well order on $\Sigma$).
\end{df}

Note that the semi-completion $\mathcal P'$ of the core presentation $\mathcal P$ is determined uniquely once the well-order on $E$ is fixed.


The following lemma describes left derivations over the semi-completion $\mathcal P'$ which start from left divisors of the distinguished edge $\rho$.

\begin{lemma}\label{left edges}
	Let $\mathcal P=\la E\mid R\ra$ be the core rewriting system of some subgroup of $F$ and let $\rho$ be the distinguished edge of the core.
	Let $\mathcal P'=\la E\mid R'=R\cup R_\iota\cup R_\tau\ra$ be the semi-completion of $\mathcal P$ (with respect to some fixed well order on $E$). Let $u$ be a non-reduced left divisor of $\rho$ in $\mathcal P'$.  Let $u'$, $e$ and $u''$ be such that $u\equiv u'eu''$ , where $u'$ is the longest reduced prefix of $u$ and $e\in E$. Then the following assertions hold.
	\begin{enumerate}
		\item[(1)]
		The unique outgoing principal left edge of $u$ is
		$$e^\ell_{u}=(\iota_\iota,\iota_{e_-}e\to\iota_{e_+},u'').$$
		\item[(2)] All the rewriting rules applied in the left derivation from $u$ to $\bar{u}$ belong to $R_\iota$. In addition, the length of the left derivation from $u$ to $\bar{u}$ is bounded from above by the length of the suffix $eu''$.
	\end{enumerate}
\end{lemma}

\begin{proof}
	(1)	Since $u'$ is a reduced left divisor of $\rho$, by Lemma \ref{left divisor}, we have $u'\equiv \iota_{u'_+}\equiv\iota_{e_-}$, so that $u\equiv \iota_{e_-}eu''$.
	Since $u'e\equiv\iota_{e_-}e$ is not reduced by assumption, $\iota_{e_-}e\to \iota_{e_+}$ is a rewriting rule in $R_\iota$. Hence, the edge  $(\iota_\iota,\iota_{e_-}e\to\iota_{e_+},u'')$ is an outgoing edge of $u$. Note that this edge 
	satisfies all three conditions of Definition \ref{df:principal}. Indeed, Condition (1) holds since every strict prefix of $\iota_{e_-}e$ is reduced. Condition (2) holds since $\iota_{e_-}e$ is the longest suffix of itself and Condition (3) holds since $\iota_{e_+}$ is the smallest word (in the ShortLex order) equivalent to $\iota_{e-}e$ over $\mathcal P'$ and since the relation $\iota_{e_-}e\to\iota_{e_+}\in R'$. Hence, the edge $(\iota_\iota,\iota_{e_-}e\to\iota_{e_+},u'')$ is the unique outgoing  principal left edge of $u$.
	
	(2) Since the left derivation from $u$ to $\bar{u}$ only visits left divisors of $\rho$, it follows from Part (1) that all the rewriting rules applied in the derivation are rewriting rules from $R_\iota$.
	
	Now, let $d^\ell(u,\bar{u})$ be the left derivation from $u$ to $\bar{u}$. We claim that  $|d^\ell(u,\bar{u})|\leq|eu''|$ where $u\equiv u'eu''$ and $u'$ is the longest reduced prefix of $u$.  By Part (1), $e^\ell_{u}=(\iota_\iota,\iota_{e-}e\to\iota_{e_+},u'')$. Hence,  $(e^\ell_u)_+=\iota_{e_+}u''$. Let $v=(e^\ell_u)_+$. Since $\iota_{e_+}$ is reduced over $\mathcal P'$,
	if one removes the longest reduced prefix of $v$, the length of the remaining suffix is at most $|u''|$. Hence, by induction, $|d^\ell(v,\bar{v})|\leq |u''|$. Since
	$|d^\ell(u,\bar{u})|=1+|d^\ell(v,\bar{v})|$ (note that $\bar{u}\equiv\bar{v}$), we have that $|d^\ell(u,\bar{u})|\leq |u''|+1=|eu''|$, as required.
\end{proof}

We will also need the following right-left analog of Lemma \ref{left edges}(2). 

\begin{lemma}\label{right der}
	Let $\mathcal P=\la E\mid R\ra$ be the core rewriting system of some subgroup of $F$ and let $\rho$ be the distinguished edge of the core.
	Let $\mathcal P'=\la E\mid R'=R\cup R_\iota\cup R_\tau\ra$ be the semi-completion of $\mathcal P$ (with respect to some fixed well order on $E$). Let $u\equiv wv$ be a right divisor of $\rho$ in $\mathcal P'$, where $v$ is the longest suffix of $u$ that is reduced over $\mathcal P'$. 	
	Then $|d^r(u,\bar{u})|\leq|w|$.
\end{lemma}

Note that if $\mathcal P=\la E\mid R\ra$ is finite, then the sets $R_\iota$ and $R_\tau$ are finite (and hence, the semi-completion $\mathcal P'$ is finite). In fact, we have the following.

\begin{lemma}\label{euler}
	Let $H$ be a closed subgroup of $F$ with a finite core $\mathcal C(H)$. Let $\rho$ be the distinguished edge of the core and let $\mathcal P'=\la E\mid R'=R\cup R_\iota\cup R_\tau\ra$ be the semi-completion of $\mathcal P$ (with respect to some appropriate well-order on $E$).
	Let $n$ be the number of inner vertices in the core $\mathcal C(H)$ (i.e., vertices other than $\iota$ and $\tau$). Let $m$ be the number of inner edges in the core (i.e., edges other than $\rho$) and let $f$ be the number of cells in the core. Then $|R_\iota|=|R_\tau|=m-n$. Hence, $|R'|\leq 2(m-n)+f$.
\end{lemma}

\begin{proof}
	Let us consider the set $R_\iota$.
	The number of rewriting rules in $R_\iota$ is equal to the number of edges $e$ such that $\iota_{e_-}e\not\equiv \iota_{e_+}$.
	Note that the total number of edges in the core is $m+1$. Hence, it suffices to prove that the number of edges $e$ for which $\iota_{e_-}e\equiv \iota_{e_+}$ is equal to $n+1$.
	
	Note that for any edge $e$, we have $\iota_{e_-}e\equiv \iota_{e_+}$ if and only if $e$ is the last edge of $\iota_{e_+}$. Indeed, if $e$ is the last edge of $\iota_{e_+}$ then $\iota_{e_+}$ has a prefix $u$ such that $\iota_{e_+}\equiv ue$. Note that $u_-=\iota$ and $u_+=e_-$ and therefore, $\iota_{e_-}\sim u$ over $\mathcal P'$. Since $u$ must be reduced (and $u$ is a left divisor of $\rho$), by Lemma \ref{left divisor}, we have $u\equiv\iota_{e_-}$ and hence $\iota_{e_+}\equiv ue\equiv \iota_{e_-}e$ as required. The opposite direction is clear.
	
	Note that the number of edges $e$ for which $e$ is the last edge of $\iota_{e_+}$ is equal to the number of vertices $\nu\neq \iota$ (for which the path $\iota_\nu$ is not empty and in particular, has a terminal edge).  
	
	Hence, the number of edges $e$ for which $\iota_{e_-}e\equiv \iota_{e_+}$ is equal to the number of vertices $\nu\neq \iota$, which is $n+1$. Hence, $|R_\iota|=(m+1)-(n+1)=m-n$ as required.
	
	The proof for $R_\tau$ is similar. Since $|R|=f$, we have that $|R'|\leq f+2(m-n)$.
\end{proof}

\begin{lemma}\label{gen_set}
	Let $\mathcal P=\la E\mid R\ra$ be a  core rewriting system with distinguished edge $\rho$. Let $\mathcal P'=\la E\mid R'=R\cup R_\iota \cup R_\tau\ra $ be the semi-completion of $\mathcal P$ (associated with some appropriate linear order on $E$). Then the set
	$$X=\{(\iota_{\ell_-},\ell\to r,\tau_{r_+})\mid \ell\to r\in R'\setminus R_\iota\}$$
	is a generating set of the fundamental group $\pi_1(K(\mathcal P'),\rho)$ of the squire complex of $\mathcal P'$.
\end{lemma}

\begin{proof}
	Since $\mathcal P'$ is a semi-complete rewriting system, Theorem \ref{X} applies to it.
	We claim that the set $X$ in this theorem and the generating set $X$ from Theorem \ref{X} coincide. In other words, we claim that the set
		$$X=\{(\iota_{\ell_-},\ell\to r,\tau_{r_+})\mid \ell\to r\in R'\setminus R_\iota\}$$
	is the set of all positive edges $(u,\ell\to r,v)$ of $K(\mathcal P')$ which satisfy the following conditions
	\begin{enumerate}
		\item[(1)] $u\ell v\sim\rho$ in $\mathcal P'$.	
		\item[(2)] $u$ and $v$ are reduced words over $\mathcal P'$.
		\item[(3)] The edge $e=(u,\ell\to r,v)$ is not a  principal left edge.
	\end{enumerate}
	Let $e=(u,\ell\to r,v)$ be a positive edge of $K(\mathcal P')$ (so that $\ell\to r\in R'$). Then $u\ell v\sim \rho$ in $\mathcal P'$ if and only if $u\ell v$ is a directed path in the core
 from $\iota$ to $\tau$. 
	Note that in that case, we must have $u_-=\iota$, $u_+=\ell_-$, $v_-=\ell_+$ and $v_+=\tau$. Hence, by Lemmas \ref{left divisor} and \ref{right divisor}, $u$ and $v$ are reduced if and only if  $u\equiv \iota_{\ell_-}$ and $v\equiv \tau_{\ell_+}$. It follows that the tuples which satisfy the first two conditions above are the tuples of the form $(\iota_{\ell_-},\ell\to r, \tau_{\ell_+})$, where $\ell\to r\in R'$.
	To complete the proof, it suffices to prove that a tuple of the form $(\iota_{\ell_-},\ell\to r, \tau_{\ell_+})$ satisfies the third condition above, if and only if $\ell\to r$ does not belong to $R_\iota$.
	Indeed, if $\ell\to r$ belongs to $R_\iota$, then for some edge $e$ of the core,
	$\ell\equiv \iota_{e_-}e$ and $r\equiv \iota_{e_+}$. Then, the edge
	$$(\iota_{\ell_-},\ell\to r, \tau_{\ell_+})=(\iota_\iota,\iota_{e_-}e\to\iota_{e_+},\tau_{e_+}).$$
	By lemma \ref{left edges}, this is the outgoing principal left edge of $\iota_{e_-}e\tau_{e_+}$ (note that $\iota_{e_-}e$ is necessarily non-reduced here).
	In the other direction, assume that  $(\iota_{\ell_-},\ell\to r, \tau_{\ell_+})$ is a  principal left edge. Let $w\equiv \iota_{\ell_-}\ell\tau_{\ell_+}$ and note that $w\sim \rho$ over $\mathcal P'$. Since $(\iota_{\ell_-},\ell\to r, \tau_{\ell_+})$ is the outgoing  principal left edge of $w$, by Lemma \ref{left edges}, the rewriting rule $\ell\to r$ must belong to $R_\iota$.
	
	Hence, the generating set of $\pi_1(K(\mathcal P'),\rho)$ constructed in Theorem \ref{X}, coincides with the set $X$ defined in this theorem.
\end{proof}



\begin{cy}\label{finitely generated}
	Let $H$ be a closed subgroup of $F$ with a finite core $\mathcal C(H)$ and let $\rho$ be the distinguished edge of the core.
	Let $\mathcal P=\la E\mid R\ra$ be the core presentation of $H$ and let $\mathcal P'=\la E\mid R'=R\cup R_\iota\cup R_\tau\ra$ be the semi-completion of $\mathcal P$ (with respect to some appropriate linear-order on $E$).
	
Then the diagram group $DG(\mathcal P',\rho)$ is finitely generated. In fact, it has a generating set of size at most $m-n+f$, where $n$ is the number of inner vertices in the core $\mathcal C(H)$, $m$ is the number of inner edges in the core and $f$ is the number of cells in the core.
\end{cy}

\begin{proof}
	Recall that the diagram group $DG(\mathcal P',\rho)$ is isomorphic to the fundamental group $\pi_1(K(\mathcal P),\rho)$. Hence, by Lemma \ref{gen_set}, it has a generating set of size $|R\setminus R_\iota|\leq |R\cup R_\tau|\leq f+m-n$, by Lemma \ref{euler}.
\end{proof}

Lemma \ref{gen_set}  shows that if the core of $H$ is finite then the diagram group $DG(\mathcal P',\rho)$ is finitely generated.
To show that $H\cong DG(\mathcal P,\rho)$ is finitely generated when the core of $H$ is finite, we make use of the following lemma. The lemma is a special case of \cite[Theorem 7.7]{GuSa}.

\begin{lemma}[\cite{GuSa}]\label{retract}
		Let $\mathcal P=\la E\mid R\ra$ and $\mathcal P'=\la E\mid R'\ra$ be string rewriting systems such that $R\subseteq R'$ and such that for every relation $\ell\to r\in R'$ the words $\ell$ and $r$ are equivalent over $\mathcal P$. Let $\rho\in E^*$.  Then the diagram group $DG(\mathcal P,\rho)$ is a retract of $DG(\mathcal P',\rho)$.
	\end{lemma}
	
	Note that the retract map (described in the proof of \cite[Theorem 7.7]{GuSa}) can be defined as follows. Let $\ell\to r\in \mathcal P'$ be a rewriting rule of the completion $\mathcal P'$. By definition, there is a derivation over $\mathcal P$ from $\ell$ to $r$. Therefore, there is an $(\ell,r)$-diagram over $\mathcal P$. For each rewriting rule $\ell\to r$ of $\mathcal P'$ we fix a diagram $\Delta(\ell\to r)$ over $\mathcal P$, where if $\ell\to r$ is a rewriting rule of $\mathcal P$ we let $\Delta(\ell\to r)$ be the atomic $(\ell,r)$-diagram over $\mathcal P$.
	Now, given a $(\rho,\rho)$-diagram $\Delta$ over $\mathcal P'$, one can replace each $(\ell,r)$-cell of $\Delta$ by $\Delta(\ell\to r)$ and each $(r,\ell)$-cell of $\Delta$ by $\Delta(\ell\to r)^{-1}$ to obtain a $(\rho,\rho)$-diagram over $\mathcal P$. This mapping from $DG(\mathcal P',\rho)$ to $DG(\mathcal P,\rho)$ is a retract. Hence, we get the following.
	
	\begin{theorem}
	Let $H$ be a closed subgroup of $F$ with a finite core $\mathcal C(H)$.
Let $\mathcal P=\la E\mid R\ra$ be the core presentation of $H$ and let $\rho$ be the distinguished edge of the core. 
Then $H\cong DG(\mathcal P,\rho)$ is finitely generated. In fact, it has a generating set of size at most $m-n+f$, where $n$ is the number of inner vertices in the core $\mathcal C(H)$, $m$ is the number of inner edges in the core and $f$ is the number of cells in the core.
	\end{theorem}

\begin{rk}\label{Y}
	In fact, we have described an algorithm for finding a finite generating set of $DG(\mathcal P,\rho)$, when the core presentation $\mathcal P=\la E\mid R\ra$ is finite:
	\begin{enumerate}
		\item[(1)] Fix a linear order on the set of edges $E$, such that the distinguished edge $\rho$ is the smallest edge.
		\item[(2)] Construct the sets of rewriting rules $R_\iota$ and $R_\tau$ to get the semi-completion $\mathcal P'$ of $\mathcal P$. (As noted above, step $1$ completely determines the semi-completion $\mathcal P'$).
		\item[(3)] The set $X$ from Lemma \ref{gen_set} is a generating set of $K(\mathcal P',\rho)\cong DG(\mathcal P',\rho)$. (We will denote the corresponding generating set of $ DG(\mathcal P',\rho)$ by $X$ as well.)
		\item[(4)] Apply a retract map $\theta$  from $DG(\mathcal P',\rho)$ to $DG(\mathcal P,\rho)$ (as described above) to get the finite generating set $Y=\theta(X)$ of $DG(\mathcal P,\rho)$.
	\end{enumerate}	
\end{rk}

	Note that if $H$ is a finitely generated subgroup of $F$ then by construction, its core is finite. Hence, we have the following

\begin{cy}
	Let $H$ be a closed subgroup of $F$. Then $H$ is finitely generated if and only if its core is finite.
\end{cy}

In particular, by Remark \ref{fin_core}, if $H$ is a finitely generated subgroup of $F$, then the closure of $H$ is also finitely generated.

%

The next lemma is the main ingredient in the proof that 
 every finitely generated closed subgroup of $F$ has linear distortion in $F$.

\begin{lemma}\label{N(delta)}
	Let $H$ be a finitely generated closed subgroup of $F$. Let $\mathcal P=\la E\mid R\ra$ be the core presentation of $H$ and let $\rho$ be the distinguished edge of the core.
	Let $\mathcal P'=\la E\mid R'=R\cup R_\iota\cup R_\tau\ra$ be the semi-completion of $\mathcal P$ (with respect to some appropriate fixed linear order). Let $X$ be the generating set of $\pi_1(K(\mathcal P'),\rho)\cong DG(\mathcal P',\rho)$ from Lemma \ref{gen_set} and let $Y=\theta(X)$ be the generating set of $DG(\mathcal P,\rho)$ obtained by applying a retract from $DG(\mathcal P',\rho)$ to $DG(\mathcal P,\rho)$ to the generating set $X$.
	Then for any diagram $\Delta$ in $DG(\mathcal P,\rho)$, the word length of $\Delta$ with respect to the generating set $Y$ is at most $3\mathcal N(\Delta)$, where $\mathcal N(\Delta)$ is the number of cells in $\Delta$.
\end{lemma}

\begin{proof}
	Let $\Delta$ be a $(\rho,\rho)$-diagram over $\mathcal P$. Since $\mathcal P$ is a tree rewriting system, the horizontal path of $\Delta$ separates it into two subdiagrams $\Delta^+$ and $\Delta^-$ such that $\Delta=\Delta^+\circ \Delta^-$.
	Let us denote by $u$ the label of the horizontal path of $\Delta$.
	
	One can enumerate the cells in $\Delta^+$ and in $(\Delta^-)\iv$ according to the right to left enumeration described in Definition \ref{rl}. Let $\Psi_1,\dots,\Psi_n$ and $\Phi_1,\dots,\Phi_n$ be the atomic diagrams associated with the enumerated cells (as in Definition \ref{rl}), such that $\Delta^+=\Psi_1\circ\cdots\circ \Psi_n$ and $(\Delta^-)\iv=\Phi_1\circ\cdots \Phi_n$.
	In particular, $$\Delta=(\Psi_1\circ\cdots\circ \Psi_n)\circ(\Phi_1\circ\cdots\circ\Phi_n)\iv.$$
	
	We will prove that the word length of $\Delta$, viewed as a diagram in $DG(\mathcal P',\rho)\cong \pi_1(K(\mathcal P'),\rho)$, with respect to the generating set $X$ is at most $3\mathcal N(\delta)$. (Since $\theta(\Delta)=\Delta$ and $\theta(X)=Y$ that would imply that the word length of $\Delta$ with respect to $Y$ is also at most $3\mathcal N(\Delta)$, as required.)
	
	For each $i$, let $e_i=(u_i,r_i\to\ell_i,v_i)$ be the negative edge of $K(\mathcal P')$ corresponding to the atomic diagram $\Psi_i$  and let $e_i'=(u_i',r_i'\to \ell_i',v_i')$ be the negative edge of $K(\mathcal P')$ corresponding to the atomic diagram $\Phi_i$,
	as in Definition \ref{rl} . Note that for all $i$, $\rho\sim u_i\ell_i v_i\sim u_i'\ell_i'v_i'$. Note also that the path $$p=e_1\cdots e_n\cdot(e_n')\iv\cdots (e_1')\iv$$ is a path on $K(\mathcal P')$ from the vertex $\rho$ to itself such that $\delta(p)=\Delta$ (recall that $\delta$ is the natural map from paths on $K(\mathcal P')$ to diagrams over $\mathcal P'$).
	
	Let us consider the equivalence class $[p]$ in $\pi_1(K(\mathcal P'),\rho)$. First, note that $$[p]=[e_1]\cdots[e_n]\cdot[e_n']\iv\cdots[e_1']\iv .$$
	
	We will prove that the word length of $[e_1]\cdots [e_n]$ with respect to $X$ is at most $\frac{3}{2}\mathcal N(\Delta)$. Since the same argument works for $[e_1']\cdots[e_n']$, that would imply the result.
	
	From the order the cells were enumerated we have that for each $1\leq i<n$, the word $v_i$ is a suffix of $v_{i+1}$. For each $1\leq i< n$, let $s_{i+1}$ be the prefix of $v_{i+1}$ such that $v_{i+1}\equiv s_{i+1}v_i$ (note that $s_{i+1}$ can be empty).
	
	For each $1\leq i\leq n$, by Corollary \ref{red}, we have $$[e_i]=[(u_i,r_i\to \ell_i,v_i)]=[(\bar{u}_i,r_i\to \ell_i,\bar{v}_i)]^{[((u_i\ell_i)*d^r(v_i,\bar{v}_i))\iv]}$$
	(Note that Corollary \ref{red} works also when $\ell\to r$ is replaced on both sides by $r\to\ell$, since  that amounts to taking the inverse of both sides.)
	
	For each $i=1,\dots,n$, let us denote $$q_i=(u_i\ell_i)*d^r(v_i,\bar{v}_i).$$
	
	Note that for each $1\leq i<n$, we have  $u_i\ell_i\sim u_{i+1}\ell_{i+1}s_{i+1}$ over $\mathcal P'$.
	Indeed,  $ u_{i+1}\ell_{i+1}s_{i+1}v_i\equiv u_{i+1}\ell_{i+1}v_{i+1}\sim \rho\sim u_i\ell_iv_i$. Hence, $u_{i+1}\ell_{i+1}s_{i+1}$ and $u_i\ell_i$ are both directed paths on the core from $\iota$ to $(v_i)_-$ and by Theorem \ref{paths},  $u_i\ell_i\sim u_{i+1}\ell_{i+1}s_{i+1}$ over $\mathcal P'$.
	Hence, by Lemma \ref{red_edges}(1), the following holds.
	$$[q_i]=[(u_i\ell_i)*d^r(v_i,\bar{v}_i)]=[(u_{i+1}\ell_{i+1}s_{i+1})*d^r(v_i,\bar{v}_i)]$$
	(indeed, the effect of replacing $u_i\ell_i$ by $u_{i+1}\ell_{i+1}s_{i+1}$ is equivalent to replacing the first coordinate of each edge of $(u_i\ell_i)*d^r(v_i,\bar{v}_i)$, by an equivalent word over $\mathcal P'$.) Note that
	\begin{equation*}
	\begin{split}
	[q_i] &=[(u_{i+1}\ell_{i+1}s_{i+1})*d^r(v_i,\bar{v}_i)]\\
	& =[(u_{i+1}\ell_{i+1})*(s_{i+1}*d^r(v_i,\bar{v}_i))]\\
	&=[(u_{i+1}\ell_{i+1})*d^r(s_{i+1}v_i,s_{i+1}\bar{v}_i)]
	\end{split}
	\end{equation*}

	Now, let $1\leq i<n$ and let us consider $q_{i+1}$.
	\begin{equation*}
	\begin{split}
[q_{i+1}]&=[(u_{i+1}\ell_{i+1})*d^r(v_{i+1},\overline{v_{i+1}})]\\
&=[(u_{i+1}\ell_{i+1})*d^r(s_{i+1}v_{i},\overline{s_{i+1}v_i})],
	\end{split}
	\end{equation*}
	where the last equality holds since $v_{i+1}\equiv s_{i+1}v_i$.
	
	Let us define, for each $1\leq i<n$,
	$$q_{i\to i+1}=(u_{i+1}\ell_{i+1})*d^r(s_{i+1}\bar{v}_i,\overline{s_{i+1}v_i}).$$
	(Note that since the reduced word equivalent to $s_{i+1}\bar{v}_i$ over $\mathcal P'$ is $\overline{s_{i+1}v_i}$, the right derivation $d^r(s_{i+1}\bar{v}_i,\overline{s_{i+1}v_i})$ is well defined).
	
	Since $d^r(s_{i+1}v_i,s_{i+1}\bar{v}_i)$ is the right derivation from $s_{i+1}v_i$ to $s_{i+1}\bar{v}_i$ and $d^r(s_{i+1}\bar{v}_i,\overline{s_{i+1}v_i})$ is the right derivation from $s_{i+1}\bar{v}_i$ to $\overline{s_{i+1}v_{i}}$, the concatenation
	$$d^r(s_{i+1}v_i,s_{i+1}\bar{v}_i)\cdot d^r(s_{i+1}\bar{v}_i,\overline{s_{i+1}v_i})$$
	is a right derivation from $s_{i+1}v_i$ to $\overline{s_{i+1}v_i}$. Indeed, it is a directed path in $K(\mathcal P')$ from $s_{i+1}v_i$ to $\overline{s_{i+1}v_i}$ which consists entirely of principal right edges. The uniqueness of the right derivation from $s_{i+1}v_i$ to $\overline{s_{i+1}v_i}$ implies that
		$$d^r(s_{i+1}v_i,s_{i+1}\bar{v}_i)\cdot d^r(s_{i+1}\bar{v}_i,\overline{s_{i+1}v_i})=d^r(s_{i+1}v_i,\overline{s_{i+1}v_i}).$$
		Hence,
		\begin{equation*}
		\begin{split}
		[q_i]\cdot[q_{i\to i+1}]&=[(u_{i+1}\ell_{i+1})*d^r(s_{i+1}v_i,s_{i+1}\bar{v}_i)]\cdot [(u_{i+1}\ell_{i+1})*d^r(s_{i+1}\bar{v}_{i},\overline{s_{i+1}v_i})]\\
		&=[(u_{i+1}\ell_{i+1})*(d^r(s_{i+1}v_i,s_{i+1}\bar{v}_i)\cdot d^r(s_{i+1}\bar{v}_i,\overline{s_{i+1}v_i}))]\\
		&=[(u_{i+1}\ell_{i+1})*d^r(s_{i+1}v_i,\overline{s_{i+1}v_i})]=[q_{i+1}]	
		\end{split}
		\end{equation*}
	Hence, for all $1\leq i<n$, we have
	$$[q_i\iv] [q_{i+1}]=[q_{i\to i+1}].$$
	
	Now, for each $i$, recall that
	$$[e_i]=[(u_i,r_i\to \ell_i,v_i)]=[(\bar{u}_i,r_i\to \ell_i,\bar{v}_i)]^{[((u_i\ell_i)*d^r(v_i,\bar{v}_i))\iv]}.$$
	Let us denote by $\overline{e_i}=(\bar{u}_i,r\to\ell,\bar{v}_i)$ and note that for each $i=1,\dots, n$,  $[\overline{e_i}]$ is either trivial or belongs to $X\iv$. 
	
	Now, for each $i=1,\dots,n$,
	$$[e_i]=[\overline{e_i}]^{[q_i^{-1}]}$$
	Hence,
	\begin{equation*}
	\begin{split}
	[e_1][e_2]\cdots [e_n] & =[\overline{e_1}]^{[q_1\iv]}[\overline{e_2}]^{[q_2\iv]}\cdots [\overline{e_i}]^{[q_i\iv]}
	[\overline{e_{i+1}}]^{[q_{i+1}\iv]}\cdots [\overline{e_n}]^{[q_n\iv]}\\
	& =[q_1][\overline{e_1}][q_1\iv][q_2][\overline{e_2}][q_2\iv]\cdots [q_i][\overline{e_i}][q_i\iv][q_{i+1}][\overline{e_{i+1}}][q_{i+1}\iv]\cdots [q_n][\overline{e_n}][q_n\iv]\\
	& =[q_1][\overline{e_1}][q_{1\to 2}][\overline{e_2}][q_{2\to 3}]\cdots [\overline{e_i}][q_{i\to i+1}][\overline{e_{i+1}}][q_{i+1\to i+2}]\dots [q_{(n-1)\to n}][\overline{e_n}][q_n\iv]
	\end{split}
	\end{equation*}
	Since for each $i=1,\dots,n$, $[\overline{e_i}]$ is either trivial or belongs to $X\iv$, 
	the total word length of $[\overline{e_1}],\dots,[\overline{e_n}]$ with respect to $X$ (i.e., the sum of their  word-lengths) is at most $n$. It suffices to prove that the total word-length of $[q_1],[q_{1\to 2}],\dots,[q_{(n-1)\to n}],[q_n\iv]$ with respect to $X$ is  bounded from above by $2n$.
	Indeed, in that case, the word length of $[e_1]\cdots[e_n]$ is bounded from above by $3n= \frac{3}{2}\mathcal N(\Delta)$, as required (note that $n$ is half the number of cells in $\Delta$). 

	First, note that the sum of lengths   of the directed paths $q_1,q_{1\to 2},\dots,q_{(n-1)\to n},q_n^{-1}$ on $K(\mathcal P)$ is at most $2|v_n|$. Indeed, it follows from Lemma \ref{right der}, that for each $i=1,\dots,n$, the length of the path $q_i$ satisfies
$$|q_i|=|(u_i\ell_i)*d^r(v_i,\bar{v}_i)|=|d^r(v_i,\bar{v}_i)|\leq |v_i|.$$
Similarly, for each $1\leq i<n$, by Lemma \ref{right der},
$$|q_{i\to{i+1}}|=|(u_{i+1}\ell_{i+1})*d^r(s_{i+1}\bar{v}_i,\overline{s_{i+1}v_i})|=|d^r(s_{i+1}\bar{v_i},\overline{s_{i+1}v_i})|\leq |s_{i+1}|,$$
Hence, $$|q_1|+\sum_{i=1}^{n-1}|q_{i\to{i+1}}|+|q_n\iv|\leq|v_1|+\sum_{i=1}^{n-1}{|s_{i+1}|}+|v_n|=|s_ns_{n-1}\cdots s_2v_1|+|v_n|=2 |v_n|$$
	Now, since the paths $q_1,q_{1\to 2},\dots,q_{(n-1)\to n},q_n$ are all right derivations over $\mathcal P'$, 
	by Lemma \ref{length}, the total word length of $[q_1],[q_{1\to 2}],\dots,[q_{(n-1)\to n}],[q_n\iv]$ with respect to $X$ is at most $2|v_n|$. It remains to note that the length of the word $v_n$ is smaller than $n$. Indeed, the horizontal path $u$ of $\Delta$ satisfies $u\equiv u_n\ell_nv_n$, since $e_n=(u_n,r_n\to\ell_n,v_n)$ corresponds to the last cell in $\Delta^+$ in the right to left enumeration. Since $|\ell_n|=2$, we have that $|v_n|\leq |u|-2=(n+1)-2=n-1$. Hence, the total word length of $[q_1],[q_{1\to 2}],\dots,[q_{(n-1)\to n}],[q_n\iv]$ is smaller than $2n$, as required. That completes the proof of the lemma.
\end{proof}

Recall that a finitely generated diagram group $G$ is said to have \emph{Property B} if the number of cells in a diagram $\Delta$ in $G$ is bi-Lipschitz equivalent to the word length of $\Delta$ (equivalently, if the word length of  a diagram $\Delta$ is at most linear in the number of cells in $\Delta$). It is well known that Thompson's group $F$, considered as a diagram group over $\la x\mid x^2\to x\ra$ has property $B$ (see, \cite{Burillo}).
It is an open question whether all finitely generated diagram groups satisfy property $B$ (see\cite[Question 1.6]{AGS}). As a corollary of Lemma \ref{N(delta)}, we get the following.


\begin{cy}\label{Property B}
	Let $\mathcal P=\la \Sigma\mid R\ra$ be a tree rewriting system and let $\rho$ in $\Sigma$. If the diagram group $DG(\mathcal P,\rho)$ is finitely generated then it has property B.
\end{cy}

\begin{proof}
	If $\mathcal P$ is a finite core rewriting system and $\rho$ is the distinguished edge of the core, the lemma follows immediately from Lemma \ref{N(delta)}.
	
	In the general case,
	consider the subgroup $H=DG(\mathcal P,\rho)$ of $F$. Since $H$ is closed and finitely generated it has a finite core rewriting system $\mathcal P_{\mathcal C}$ with a distinguished edge $\rho_{\mathcal C}$. Then the subgroup  $H$ coincides with the diagram group $DG(\mathcal P_{\mathcal C},\rho_{\mathcal C})$ (when they are viewed as subgroups of $F$). Since $DG(\mathcal P,\rho)$ and $DG(\mathcal P_{\mathcal C},\rho_{\mathcal C})$ coincide as subgroups of $F$, 	
	 there is an isomorphism from $DG(\mathcal P,\rho)$ to $DG(\mathcal P_{\mathcal C},\rho_{\mathcal C})$ which preserves the number of cells in a diagram. 
	  Hence, we are done by the previous case.
	
%
%
%
%
%
\end{proof}

Recall that if $G$ is a group generated by a finite set $S$ and $H$ is a subgroup of $G$ generated by a finite set $T$, then the distortion function $\delta_{S,T}$ is the smallest function $\N\to \N$ such that if an element $h\in H$ is a product of $n$ elements of $S$, then it is a product of at most $\delta_{S,T}(n)$ elements of $T$. For fixed $G, H$ but different (finite) $S, T$, the functions $\delta_{S,T}$ are equivalent\footnote{Two functions $f,g\colon \N\to \N$ are called \emph{equivalent} if for some $c>1$, $\frac 1c f(\frac nc)-c \le g(n)\le cf(cn)+c$ for every $n\in \N$.}.  The subgroup $H$ is called \emph{undistorted} in $G$ if the distortion function is linear. 
Although many subgroups of  Thompson's group $F$ are known to be undistorted (see, for example, \cite{GuSa,Burillo, GubSa1, WC,GS2}), $F$ has distorted subgroups \cite{GubSa1, DO}. Corollary \ref{Property B} and the fact that $F$ has property $B$ as a diagram group over $\la x\mid x^2\to x\ra$  imply the following.

\begin{theorem}\label{undis}
	Let $H$ be a finitely generated closed subgroup of $F$.
	Then $H$ has linear distortion in Thompson's group $F$. \qed
\end{theorem}

\section{Examples}\label{sec:exa}

Theorem \ref{undis} implies that many subgroups of $F$ are undistorted. We give here a few examples. In most of the examples, it is convenient to view $F$ as a group of homeomorphisms of $[0,1]$. Hence, we begin by recalling the definition of $F$ as a group of homeomorphisms, the relation to $F$ viewed as a diagram group and some related notions.

\subsection{$F$ as a group of homeomorphisms}

Recall that Thompson's group $F$ is the group of all piecewise linear homeomorphisms of the interval $[0,1]$ where all breakpoints are finite dyadic and all slopes are integer powers of $2$.
The group $F$ is generated by two functions, usually denoted $x_0$ and $x_1$ (see \cite{CFP}).
The composition in $F$ is from left to right.

Every element of $F$ is completely determined by how it acts on the set $\zz$. Every number in $(0,1)$ can be described as $.s$ where $s$ is an infinite word in $\{0,1\}$. For each element $g\in F$ there exists a finite collection of pairs of finite binary words $(u_i,v_i)$ such that every infinite binary word  starts with exactly one of the $u_i$'s. The action of $F$ on a number $.s$ is the following: if $s$ starts with $u_i$, we replace $u_i$ by $v_i$. For example, the generators $x_0$ and $x_1$ of $F$  are the following functions:
\[
x_0(t) =
\begin{cases}
.0\alpha &  \hbox{ if }  t=.00\alpha \\
.10\alpha       & \hbox{ if } t=.01\alpha\\
.11\alpha       & \hbox{ if } t=.1\alpha\
\end{cases} 	\qquad	
x_1(t) =
\begin{cases}
.0\alpha &  \hbox{ if } t=.0\alpha\\
.10\alpha  &   \hbox{ if } t=.100\alpha\\
.110\alpha            &  \hbox{ if } t=.101\alpha\\
.111\alpha                      & \hbox{ if } t=.11\alpha\
\end{cases}
\]
where $\alpha$ is any infinite binary word.


\subsubsection{The relation between $F$ as a diagram group and $F$ as a group of homeomorphisms} \label{sec:red}

Instead of describing elements of $F$ as diagrams, one can describe them as pairs of full finite binary trees. Let $\Delta$ be a reduced diagram in $F$ and consider the subdiagrams $\Delta^+$ and $\Delta^-$. If one puts a vertex at the middle of every edge of $\Delta^+$ and for each cell $\pi$ in $\Delta^+$, draws an edge from the vertex on $\topp(\pi)$ to each of the vertices on $\bott(\pi)$ one gets a full finite binary tree $T_+$. If one applies the same construction to the diagram $(\Delta^-)\iv$, one gets a full finite binary tree $T_-$, such that $T_+$ and $T_-$ have the same number of leaves. The element $\Delta$ is represented by the \textit{tree-diagram} $(T_+,T_-)$.


If $T$ is a finite binary tree, a \emph{branch} of $T$ is a maximal simple path starting from the root. 
If every left edge of $T$ is labeled  $``0"$ and every right edge is labeled  $``1"$, then every branch of $T$ is labeled by a finite binary word $u$. Now, let $\Delta$ be a reduced diagram and let $(T_+,T_-)$ be the corresponding tree-diagram where $T_+$ and $T_-$ have $n$ leaves. Let $u_1,\dots,u_n$ (resp. $v_1,\dots,v_n$) be the (labels of) branches of $T_+$ (resp. $T_-$), ordered from left to right.
For each $i=1,\dots,n$ we say that the tree-diagram $(T_+,T_-)$ has the \emph{pair of branches} $u_i\rightarrow v_i$. The function $g$ from $F$ corresponding to this tree-diagram takes binary fraction $.u_i\alpha$ to $.v_i\alpha$ for every $i$ and every infinite binary word $\alpha$.

\subsubsection{Closed subgroups of $F$ in terms of homeomorphisms}

Recall that by Lemma \ref{dyadic} a subgroup $H$ of $F$ is closed if and only if every function $f\in H$ that is a piecewise-$H$ function belongs to $H$. This condition can also be described as follows.

\begin{df}
	Let $g\in F$ be a function which fixes a dyadic fraction $\alpha\in (0,1)$. The \emph{components} of $g$ at $\alpha$ are the functions.
	\[
	g_1(t) =
	\begin{cases}
	g(t) &  \hbox{ if }  t\in[0,\alpha] \\
	t       & \hbox{ if } t\in[\alpha,1]\
	\end{cases}  	\qquad	
	g_2(t) =
	\begin{cases}
	t &  \hbox{ if }  t\in[0,\alpha] \\
	g(t)       & \hbox{ if } t\in[\alpha,1]\
	\end{cases},
	\]
\end{df}

\begin{lemma}[{\cite[Corollary 5.7]{G}}]\label{components}
	Let $H$ be a subgroup of $F$. Then $H$ is a closed subgroup of $F$ if and only if for every function $h\in H$ and every dyadic fraction $\alpha$ such that $h$ fixes $\alpha$, the components of $h$ at $\alpha$ belong to $H$.
\end{lemma}

Note that Lemma \ref{components} implies that if $S$ is a subset of $[0,1]$, then the stabilizer of $S$ in $F$ is a closed subgroup of $F$. Similarly, it implies that the intersection of closed subgroups of $F$ is a closed subgroup of $F$.

\subsection{Examples of undistorted subgroups of $F$}

\paragraph{Cyclic subgroups.} Theorem \ref{undis} can be used to recover the fact that every cyclic subgroup of $F$ is undistorted in $F$ (see \cite{Burillo}). Indeed, it follows from Lemma \ref{components} that the closure of any cyclic subgroup $H$ of $F$ is abelian  and finitely generated (it also follows from \cite[Theorem 11.1]{G}). Hence, if $H$ is a cyclic subgroup of $F$ then $H$ is undistorted in $Cl(H)$ and the closure $Cl(H)$ is undistorted in $F$. Therefore, $H$ is undistorted in $F$.
Note that by Lemma \ref{components}, a cyclic subgroup $H=\la f\ra$ of $F$ is closed if and only if $f$ does not fix any dyadic fraction in $(0,1)$.



\paragraph{Solvable subgroups.} Recall that in \cite{Cleary}, it is proved that Thompson's group $F$ has an undistorted subgroup isomorphic to $\mathbb{Z}\wr\mathbb{Z}$. Let $\mathcal S$ be the  smallest class of groups that includes the group $\mathbb{Z}$ and is closed under finite direct sums and under wreath products with $\mathbb{Z}$.
Here, we demonstrate that any group in the class $\mathcal S$ can be embedded in $F$ without distortion.

Recall that an \emph{orbital} of an element $f\in F$ is an interval $(a,b)\subseteq (0,1)$ such that $f$ fixes the endpoints $a$ and $b$ and does not fix any point in $(a,b)$.
The orbital $(a,b)$ is a \emph{push-up} (resp. push-down) orbital if for every $x\in (a,b)$, we have $f(x)>x$ (resp. $f(x)<x$).
We say that a subset $A$ of the interval $[0,1]$ is a \emph{fundamental domain} of a function $f\in F$ if it contains exactly one point from each non-trivial orbit of the action of $f$ on $[0,1]$ (Note that if $f$ has a unique orbital $(a,b)$, then an interval $[c,d)\subseteq (a,b)$ is a fundamental domain of the element $f$ if $f(c)=d$). 
The \emph{support} of the function $f$ (resp. a subgroup $H\leq F$) is the set of all points $x\in (0,1)$ which are not fixed by $f$ (resp. by $H$). We say that $f$ (resp. $H$) is supported in a subset $I\subseteq (0,1)$ if the support of $f$ (resp. $H$) is contained in $I$. 



Thompson's group $F$ contains many copies of itself (see \cite{BrinU}). We will be interested in copies of the following simple form. Let $u$ be a finite binary word. We denote by $[u]$ the dyadic interval $[.u,.u1^{\mathbb{N}}]$.  We denote by $F_{[u]}$ the subgroup of $F$ of all functions supported in the interval $[u]$ (i.e., the subgroup of all functions which fix $[0,1]\setminus[u]$ pointwise).  It is easy to see that $F$ is isomorphic to $F_{[u]}$.
Indeed, Thompson's group $F$ can be viewed as a subgroup of the group $\pl_2(\mathbb{R})$ of all piecewise linear homeomorphisms of  $\mathbb R$ with a finite number of  dyadic break points and absolute values of all slopes powers of 2.  Let $\phi_u\in \pl_2(\mathbb{R})$ be an increasing function which maps the interval $[0,1]$ linearly onto $[u]$, (such a function clearly exists). Then $\phi_u$ induces an isomorphism from $F$ to $F_{[u]}$ by conjugation. Indeed, $F^{\phi_u}$ is the subgroup of $\pl_2(\mathbb{R})$  of all orientation preserving homeomorphisms with support in $[u]$, that is, $F^{\phi_u}=F_{[u]}$.
 If $g\in F$, we denote by $g_{[u]}$ its image $g^{\phi_u}$ under the above isomorphism (and note that the image does not depend on the choice of the function $\phi_u$). We refer to $g_{[u]}$ as the \emph{$[u]$-copy} of $g$.

Note that by Lemma \ref{dyadic}, for every finite binary word $u$, the subgroup $F_{[u]}$ is a closed subgroup of $F$ (and hence undistorted in $F$). Let $H$ be a subgroup of $F$. We denote by $H_{[u]}$ the image of $F$ in $F_{[u]}$ under the above isomorphism. It follows from Lemma \ref{components}, that if $H$ is a closed subgroup of $F$ then $H_{[u]}$ is also a closed subgroup of $F$.

The following observation is well known (see, for example \cite{Brin}).

\begin{lemma}\label{wreath}
	Let $H_1,H_2$ be subgroups of $F$ and let $f\in F$. Then the following assertions hold.
	\begin{enumerate}
		\item[(1)] If $H_1$ and $H_2$ have disjoint supports then $H_1$ and $H_2$ commute and the subgroup of $F$ generated by them is isomorphic to $H_1\times H_2$.
		\item[(2)] 	 Assume that the support of $H$ is contained in a fundamental domain of $f$. Then the subgroup of $F$ generated by $H$ and $f$ is isomorphic to the wreath product $H\wr \mathbb{Z}$.
	\end{enumerate}
\end{lemma}

We will need the following lemma.

\begin{lemma}\label{wreath closed}
	Let $H_1$, $H_2$ be closed subgroups of $F$. Then the following assertions hold.
	\begin{enumerate}
		\item[(1)] Thompson's group $F$ has a closed subgroup isomorphic to $H_1\times H_2$.
		\item[(2)] Thompson's group $F$ has a closed subgroup isomorphic to $H_1\wr\mathbb{Z}$.
	\end{enumerate}
\end{lemma}

\begin{proof}
	We only prove part (2).  Let $f$ be an element of $F$ such that $f$ has a unique (push-up) orbital $(a,b)$ and let $u$ be a finite binary word such that the dyadic interval $[u]$ is contained in a fundamental domain of $f$. Let $H=(H_1)_{[u]}$ be the image of $H$ in $F_{[u]}$ under the natural isomorphism defined above. Let $G$ be the subgroup of $F$ generated by $f$ and $H$.
	 By Lemma \ref{wreath}, $G\cong H_1\wr\mathbb{Z}$. It suffices to prove that $G$ is closed. The proof that $G$ is closed is similar to the proof of \cite[Corollary 4.6]{BBH}.
	
%
	 Let $g\in G$ and assume that $g$ fixes a dyadic fraction $\alpha\in (0,1)$. By Lemma \ref{components} it suffices to prove that the components of $g$ at $\alpha$ belong to $G$.
	  From the structure of $G$ as a wreath product, there exist integers $r\geq 0$ and $m\in\mathbb{Z}$, elements $h_1, \dots,h_r\in H$ and integers $k_1<\cdots<k_r$ such that
	 $$g=h_1^{f^{k_1}}h_2^{f^{k_2}}\cdots h_r^{f^{k_r}}f^m$$
	
	  Note that for each $i$, the element $h_i$ is supported in the interval $[u]$. Hence, the element $h_i^{f^{k_i}}$ is supported in $f^{k_i}([u])$.  Since $[u]$ is contained in a fundamental domain of $f$, the supports of the elements $h_1^{f^{k_1}},h_2^{f^{k_2}},\cdots, h_r^{f^{k_r}}$ are disjoint. 
	 Then, since $(a,b)$ is a push-up orbital, for each $i<j$ in $\{1,\dots,r\}$, the support of  $h_i^{f^{k_i}}$ is entirely to the left of the support of $h_j^{f^{k_j}}$.
	
	 Now, assume first that $m=0$. In that case,
	 $$g=h_1^{f^{k_1}}h_2^{f^{k_2}}\cdots h_r^{f^{k_r}}$$
	  Let $i\in\{1,\dots,r\}$ be the minimal index such that the support of $h_i^{f^{k_i}}$ is not contained in $(0,\alpha)$. 
	   (If no such index exists, then the components of $g$ at $\alpha$ are $g$ and the identity and we are done.) If the support of $h_i^{f^{k_i}}$ is contained in $[\alpha,1]$, then the components of $g$ at $\alpha$ are $$g_1=h_1^{f^{k_1}}\cdots h_{i-1}^{f^{k_{i-1}}}\ \mbox{and }\ g_2=h_i^{f^{k_i}}\cdots h_r^{f^{k_r}}.$$
	 In particular, they both belong to $G$, as required.
	 If the support of $h_i^{f^{k_i}}$ is not contained in $[\alpha,1]$, we consider the binary fraction $\alpha'=f^{-k_i}(\alpha)$. Note that $h_i^{f^{k_i}}$ fixes $\alpha$ (since $\alpha$ is a fixed point of $g$). Hence,  $\alpha'$ is a fixed point of $h_i$. Since $h_i\in H$ and $H$ is a closed subgroup of $F$, the components of $h_i$ at $\alpha'$ belong to $H$. As the components $\xi_1$ and $\xi_2$ of $h_i^{f^{k_i}}$ at $\alpha$ are $f^{k_i}$-conjugates of the components of $h_i$ at $\alpha'$, they belong to $G$. Then, the components of $g$ at $\alpha$ are
	 $$g_1=h_1^{f^{k_1}}\cdots h_{i-1}^{f^{k_{i-1}}}\xi_1\ \mbox{and }\ g_2=\xi_2h_i^{f^{k_{i+1}}}\cdots h_r^{f^{k_r}},$$
	 and in particular, they belong to $G$, as required.
	
	 Hence, it suffices to consider the case where $m\neq 0$. We claim that in that case, $(a,b)$ is the unique orbital of $g$. Since the support of $f$ is $(a,b)$ and the support of each $h_i^{f^{k_i}}$ is contained in $(a,b)$, the function $g$ fixes pointwise $[0,1]\setminus (a,b)$. Let $x\in (a,b)$. It suffices to show that $f(x)\neq x$. If $x$ is not in the support of $h_i^{f^{k_i}}$ for any $i\in\{1,\dots,r\}$, then $g(x)=f^m(x)\neq x$ since $x\in (a,b)$ and $(a,b)$ is an orbital of $f$. Otherwise, there is a unique $i\in\{1,\dots,r\}$ such that $x$ is in the support of $h_i^{f^{k_i}}$. Note that, $h_i^{f^{k_i}}(x)$ is also in the support of $h_i^{f^{k_i}}$ and as such, it is fixed by $h_j^{f^{k_j}}$ for all $j\neq i$.
	 In particular, $g(x)=f^m(h_i^{f^{k_i}}(x))$. Note that, $x,h_i^{f^{k_i}}(x)\in f^{k_i}([u])$ since $h_i^{f^{k_i}}$ is supported in $f^{k_i}([u])$. Since $[u]$, and as such $f^{k_i}([u])$, is contained in a fundamental domain of $f$, $f^m(h_i^{f^{k_i}}(x))\not\in f^{k_i}([u])$. Hence, $g(x)=f^m(h_i^{f^{k_i}}(x))\neq x$ as necessary.
	
	 Now, since $(a,b)$ is the unique orbital of $g$, every component of $g$ is either trivial or coincides with $g$. In particular, it belongs to $G$.
	 Hence, $G$ is a closed subgroup of $F$.
\end{proof}


By Lemma \ref{wreath closed} and the fact that $F$ has a  closed cyclic subgroup we have the following.

\begin{cy}\label{solv}
	Any group in the class $\mathcal S$ 
	 can be embedded in $F$ as a closed subgroup, and hence, without distortion.
\end{cy}

For an example of an embedding of $\mathbb{Z}\wr\mathbb{Z}$ into $F$ as a closed subgroup (and the core rewriting system of the associated subgroup), see \cite[Example 11.27]{G}.

Note that all groups in the class $\mathcal S$
  are solvable. Recall that by \cite[Theorem 11.1]{G}, the closure of any finitely generated solvable subgroup of $F$ is finitely generated and solvable (of the same derived length). In fact, we have the following.

\begin{lemma}
	Let $H$ be a finitely generated solvable subgroup of $F$. Then the closure of $H$ in $F$ belongs to the class $\mathcal S$. 
\end{lemma}

\begin{proof}
Recall that Bleak defined the \emph{split group} $S(H)$ of a subgroup $H$ of $PL_o(I)$, as the subgroup of all piecewise-$H$ functions. In particular, if $H$ is a subgroup of $F$, then $S(H)$ is not necessarily a subgroup of $F$ but for every $H\leq F$ we have $H\leq Cl(H)\leq S(H)$.
In \cite{BBH}, Bleak Brough and Hermiller give an algorithm for determining the solvability of computable subgroups of $PL_o(I)$. It follows from the algorithm that if $H$ is a finitely generated solvable subgroup of $F$ then the split group $S(H)$ belongs to the class $\mathcal S$. Indeed, they show in the proof that $S(H)$ is generated by a set of one-orbital functions which satisfies certain conditions (see \cite[Page 4]{BBH}). They show that these conditions guarantee (by an inductive application of Lemma \ref{wreath} above) that the split group $S(H)$ belongs to the class $\mathcal S$.
The proof can be modified to show that if $H$ is finitely generated and solvable then the closure $Cl(H)$ belongs to the class $\mathcal S$.
Indeed, let $f\in F$. We say that an interval $(a,b)$ is a \emph{dyadic-orbital} of $f$ if $a$ and $b$ are dyadic, $f$ fixes $a$ and $b$ and $f$ does not fix any dyadic fraction in $(a,b)$.
The proof from \cite{BBH} can be modified to show that if $H$ is finitely generated and solvable then $Cl(H)$ is generated by a finite set $Y$ such that each function in $Y$ has a unique dyadic-orbital and such that using Lemma \ref{wreath}, one can get that the subgroup generated by $Y$ belongs to the class $\mathcal S$. Hence, $Cl(H)$ belongs to $\mathcal S$.
%
%
\end{proof}

\begin{rk}
	Note that $\mathbb{Z}\wr\mathbb{Z}$ has subgroups with distortion functions of any polynomial degree \cite{DO}.
	It follows that Thompson's group $F$ also has distorted copies of $\mathbb{Z}\wr\mathbb{Z}$. By Theorem \ref{undis}, these subgroups are necessarily not closed. Note also that their closure is finitely generated and closed, and hence undistorted in $F$. It follows that these subgroups are distorted subgroups of their closure.	
\end{rk}

\paragraph{The Brin-Navas group.}

Recall that the Brin-Navas group $B$ is an elementary amenable group of EA class $\omega+2$ (see \cite[Section 5]{Brin} and \cite[Example 6.3]{N}). The group $B$ is an HNN-extension of a bi-infinite iterated wreath product of $\mathbb{Z}$.
It is proved in \cite[Theorem 9.1]{G} that there is a closed subgroup of $F$ isomorphic to $B$. In fact, it is proved that $F$ has a closed subgroup $H$ isomorphic to $B$ which is  maximal inside a normal subgroup $K$ of $F$ (such that $F/K$ is cyclic). By Theorem \ref{undis}, this copy of the Brin-Navas group inside $F$ is undistorted in $F$. (The core rewriting system  of this copy of $B$ inside $F$ appears in the proof of \cite[Theorem 9.1]{G}).



\paragraph{Maximal subgroups of $F$.}

As noted in the introduction, all maximal subgroups of $F$ of infinite index are closed \cite{G21}. Hence, Theorem \ref{undis} implies that every finitely generated maximal subgroup of $F$ is undistorted in $F$ (note that finite index subgroups of $F$ are necessarily undistorted in $F$).




Note also that if $H$ is a finitely generated proper subgroup of $F$ then $H$ is contained in a maximal subgroup whose core is finite (see \cite{G21}). Hence, Corollary \ref{finitely generated} implies the following.

\begin{cy}
	Let $H$ be a finitely generated proper subgroup of $F$. Then $H$ is contained in a finitely generated maximal subgroup of $F$.
\end{cy}

\paragraph{Jones' subgroups of Thompson's group $F$.}

Vaughan Jones \cite{Jones} defined a family of unitary representations of Thompson's group $F$ using planar algebras. These representations give rise to interesting subgroups of $F$ (the stabilizers of the vacuum vector in these representations). Jones' subgroup $\overrightarrow{F}$, defined in \cite{Jones} is particularly interesting. Indeed, Jones proved that elements of $F$ encode in a natural way all knots and links and elements of $\overrightarrow{F}$ encode all oriented links and knots. In \cite{GS-J,GS} we proved that Jones' subgroup $\overrightarrow{F}$ is isomorphic to
the ``brother group'' $F_3$  of $F$ (i.e., the group of all piecewise linear homeomorphisms of the interval $[0,1]$, where all slopes are integer powers of $3$ and break points of the derivative are  3-adic fractions \cite{Br}).  We also showed that $\overrightarrow F$ is the stabilizer of the set $S$ of all dyadic fractions such that the sum of digits in their finite binary representation is odd.
Hence, Jones' subgroup $\overrightarrow{F}$ is a closed subgroup of $F$ and Theorem \ref{undis} implies that it is undistorted in $F$. (The interested reader can find the core rewriting system of $\overrightarrow{F}$ in \cite[Example 6.10]{G}.) In \cite{GS-J} we have also studies a family of subgroups, which we called Jones' subgroups $\overrightarrow F_n$, which can be defined in an analogous way to $\overrightarrow F$ (where $\overrightarrow F_2=\overrightarrow F$). We showed that like $\overrightarrow{F}$, for each $n$, Jones' subgroup $\overrightarrow{F}_n$ is isomorphic to the brother group $F_{n+1}$ of $F$. We also showed that for each $n$, $\overrightarrow{F}_n$ is the intersection of stabilizers of certain sets of dyadic fractions (see \cite{GS-J}). Hence they are all closed subgroups of $F$ and as such undistorted in $F$. Note that in \cite{G21}, the first author proves that for every prime number $p$, Thompson's group $F$ has a maximal subgroup isomorphic to $\overrightarrow F_p$.

Another subgroup defined by Jones, is the $3$-colorable subgroup $\mathcal F$. It was studied in \cite{Ren,AN}. It follows from \cite{AN}, that this subgroup is closed, since it can be described as the stabilizer of some set of numbers in $(0,1)$. Hence, by Theorem \ref{undisInt}, it is also undistorted in ${F}$.

\section{Open Problems}\label{open}

Using methods from \cite[Section 10]{G}, it is possible to prove that the closed subgroup of $F$ with undecidable conjugacy problem constructed in Section \ref{sec:con} is not finitely generated if the group $G$ with undecidable word problem used in the construction is torsion-free.

\begin{prob}\label{prob:conj}
	Does Thompson's group $F$ have a closed finitely generated subgroup with undecdiable conjugacy problem?
\end{prob}

Recall that in Section \ref{sec:fin} we prove that the core rewriting system of any finitely generated subgroup of $F$ has a finite semi-completion. By \cite{GuSa}, if $\mathcal P$ is a string rewriting system which has a finite completion, then the conjugacy problem in any diagram group over $\mathcal P$ is decidable. Hence, if the core rewriting system of any finitely generated subgroup of $F$ has a finite completion, then the answer to Problem \ref{prob:conj} is negative.

\begin{prob}\label{prob:comp}
	Let $H$ be a finitely generated closed subgroup of $F$ and let $\mathcal P=\la E\mid R\ra$ be the core rewriting system of $H$. Does $\mathcal P$ have a finite completion?
\end{prob}


We note also that finitely generated closed subgroups of $F$ are not necessarily finitely presented. Indeed, as noted above, Thompson's group F has a closed subgroup isomorphic to $\mathbb{Z}\wr\mathbb{Z}$ . The following problem remains open.

\begin{prob}
	Is every finitely generated maximal subgroup of Thompson's group $F$ finitely presented?
\end{prob}

Recall that Savchuk proved \cite{Sav,Sav1} that for any $\alpha\in (0,1)$ the stabilizer $S_\alpha$ of $\alpha$ in $F$ is a maximal subgroup of $F$. By \cite{GS2}, this maximal subgroup is finitely generated if and only if $\alpha$ is rational. As noted above, for every prime number $p$, Thompson's group $F$ has a maximal subgroup isomorphic to $\overrightarrow{F}_p\cong F_{p+1}$. Note that there are only six other explicit maximal subgroups of $F$ of infinite index in the literature: three maximal subgroups of $F$ were constructed in \cite[Section 10]{G} and three explicit maximal subgroups of $F$ (which contain Jones' $3$-colorable subgroup $\mathcal F$) were recently constructed in \cite{AN}.
It is possible to prove that all of these maximal subgroups of $F$ are finitely presented. Indeed, this is obvious for the maximal subgroups of $F$ isomorphic to $F_{p+1}$ (recall that by \cite{GS}, all ``brother groups'' $F_n$ of Thompson's group $F$ are finitely presented). For any other known maximal subgroup of $F$ of infinite index, one can find the core rewriting system of the subgroup and show that it has a finite completion. Then using \cite[Lemma 9.11]{GuSa}, one can show that the subgroup is finitely presented (this is also true for many maximal subgroups of $F$ which do not appear in the literature and can be constructed in a similar way to the maximal subgroups of $F$ constructed in \cite[Section 10]{G}).

Another interesting problem for closed subgroups of $F$ is the following.

\begin{prob}
	Let $H$ be a closed subgroup of $F$ and let $\mathcal P=\la E\mid R\ra$ be the core rewriting system of $H$.
	Is it true that $H$ contains a copy of Thompson's group $F$ if and only if the semigroup defined by the rewriting system $\mathcal P$ contains an idempotent?
\end{prob}

We note that if the semigroup $S$ defined by the rewriting system $\mathcal P$ contains an idempotent, then $H$ contains a copy of $F$, by \cite[Theorem 25]{GubSa1}. The other direction is open. Note also that if the answer to Problem \ref{prob:comp} is positive, then it is likely that for every finite core rewriting system it is decidable whether the semigroup it defines contains an idempotent.


\begin{minipage}{3 in}
	Gili Golan\\
Department of Mathematics,\\
 Ben Gurion University of the Negev,\\ 
golangi@bgu.ac.il
\end{minipage}
\begin{minipage}{3 in}
	Mark Sapir\\
	Department of Mathematics,\\
	Vanderbilt University,\\
	m.sapir@vanderbilt.edu
\end{minipage}

\end{document}